\documentclass[reqno]{amsart} 
\usepackage{epstopdf}
\usepackage{textcomp}
\usepackage{amsmath, amsthm, amscd, amsfonts}
\usepackage{graphicx}
\usepackage{float}
\usepackage{mathrsfs}
\usepackage{bbm}

\usepackage{hyperref}
\hypersetup{colorlinks=true, citecolor=blue}
\usepackage{color}
\usepackage[usenames,dvipsnames,svgnames,table]{xcolor}

\newcommand{\mb}{\mathbb}

\newcommand{\mf}{\mathfrak}
\newcommand{\mc}{\mathcal}
\newcommand{\ms}{\mathscr}
\newcommand{\ov}{\overline}

\newcommand{\wt}{\widetilde}

\makeatletter % `@' now normal "letter"
\@addtoreset{equation}{section}
\makeatother  % `@' is restored as "non-letter"

\numberwithin{equation}{section}

\begin{document}

\numberwithin{equation}{section}

\title{Morse theory for Lagrange multipliers and adiabatic limits}

\author[Schecter]{Stephen Schecter}
\address{
Department of Mathematics \\
North Carolina State University \\
Box 8205 \\
Raleigh, NC 27695 USA \\
+1-919-515-6533
}
\email{schecter@ncsu.edu}

\author[Xu]{Guangbo Xu}
\address{
Department of Mathematics, University of California, Irvine, Irvine, CA 92697 USA.
}
\email{guangbox@math.uci.edu}

\date{May 19, 2014}

%\received{}
%\communicated{}

%\subjclass{57R70, 37D15}

%\keywords{Morse homology, adiabatic limits, geometric singular perturbation theory, exchange lemma}

\newtheorem{thm}{Theorem}
\newtheorem*{thm-no}{Theorem}
\newtheorem*{cor-no}{Corollary}
\newtheorem{lem}[thm]{Lemma}
\newtheorem{cor}[thm]{Corollary}
\newtheorem{prop}[thm]{Proposition}

\theoremstyle{definition}
\newtheorem{defn}[thm]{Definition}
\theoremstyle{remark}
\newtheorem{rem}[thm]{Remark}
\newtheorem{hyp}[thm]{Hypothesis}

\begin{abstract}
Given two Morse functions $f, \mu$ on a compact manifold $M$, we study the Morse homology for the Lagrange multiplier function on $M \times {\mb R}$, which sends $(x, \eta)$ to $f(x) + \eta \mu(x)$. Take a product metric on $M \times {\mb R}$, and rescale its ${\mb R}$-component by a factor $\lambda^2$.  We show that generically, for large $\lambda$, the Morse-Smale-Witten chain complex is isomorphic to the one for $f$ and the metric restricted to ${\mu^{-1}(0)}$, with grading shifted by one.  On the other hand, in the limit $\lambda\to 0$, we obtain another chain complex, which is geometrically quite different but has the same homology as the singular homology of $\mu^{-1}(0)$. The isomorphism between the chain complexes is provided by the homotopy obtained by varying $\lambda$. Our proofs use both the implicit function theorem on Banach manifolds and geometric singular perturbation theory.

\end{abstract}

\maketitle

\setcounter{tocdepth}{1}

\tableofcontents

\section{Introduction\label{sec:intro}}

Let $M$ be a compact manifold. Suppose $f$ and $\mu$ are Morse functions on $M$, and $0$ is a regular value of $\mu$. 
Then we consider the Lagrange multiplier function
\begin{align*}
\begin{array}{cccc}
{\mc F}: &  M \times {\mathbb R}  & \to &  {\mathbb R},\\
        & (x, \eta)  &   \mapsto  &    f(x)+ \eta \mu(x).
        \end{array}
\end{align*}
The critical point set of ${\mc F}$ is
\begin{align}
{\rm Crit} ({\mc F})= \left\{ (x, \eta)\  |\ \mu(x)=0, d f (x) + \eta d \mu(x)=0   \right\},
\end{align}
and there is a bijection 
\begin{align*}
\begin{array}{ccc}
{\rm Crit} ({\mc F}) & \simeq & {\rm Crit} \left( f|_{\mu^{-1}(0)} \right),\\
(x,\eta) &  \mapsto & x.
\end{array}
\end{align*}
This is a topic which is taught in college calculus.

A deeper story is the Morse theory of ${\mc F}$. Take a Riemannian metric $g$ on $M$ and the standard Euclidean metric $e$ on ${\mb R}$. Denote by $g_1$ the product metric $g\oplus e$ on $M\times {\mb R}$. Then the gradient vector field of ${\mc F}$ with respect to $g_1$ 	is
\begin{align}
\nabla {\mc F}(x, \eta) = ( \nabla f + \eta \nabla \mu, \mu(x) ).
\end{align}The differential equation for the negative gradient flow of ${\mc F}$ is
\begin{align}
 x' & =  - \left( \nabla f (x)+ \eta \nabla \mu(x) \right),    \label{eqn15-1}    \\
 \eta' & =  - \mu(x). \label{eqn15-2}
\end{align}
Let $p_\pm= (x_\pm, \eta_\pm) \in {\rm Crit} ({\mc F})$, and let ${\mc M}(p_-, p_+)$ be  the space of orbits of the negative gradient flow that approach $p_\pm$ as $t\to \pm\infty$ respectively.   ${\mc M}(p_-, p_+)$ is called a {\em moduli space} of orbits. For generic choice of data $(f, \mu, g)$, this moduli space will be a smooth manifold with dimension
\begin{align}
{\rm dim}\,{\mc M}(p_-, p_+) = {\rm index} (p_-) - {\rm index} (p_+) -1.
\end{align}
When the dimension is zero, we can count the number of elements of the corresponding moduli space and use the counting to define the Morse-Smale-Witten complex.  Its homology is called the Morse homology of the pair $({\mc F},g\oplus e)$.

We shall show that if we rescale the metric on the ${\mb R}$-part, obtaining $g_\lambda: = g \oplus \lambda^{-2} e$, $\lambda\in {\mb R}^+$, then the chain complex is defined for $\lambda$ in a generic set $\Lambda^{reg}$, and for all such $\lambda$ we have isomorphic homology groups. The system (\ref{eqn15-1})--(\ref{eqn15-2}) is replaced by
\begin{align}
x' & =  - \left( \nabla f (x)+ \eta \nabla \mu(x) \right),  \label{eqn15a} \\
\eta' & =  -\lambda^2 \mu(x). \label{eqn15b}
\end{align}
We then consider the limits of the complex for $({\mc F}, g_\lambda)$ as $\lambda$ approaches  $\infty$ and zero. Orbits in the two limits will be completely different geometric objects, but the counting of them gives the same homology groups. 

The main results of this paper can be summarized as the following theorems. For precise meaning and necessary assumptions of these theorems, the reader is referred to Theorem \ref{thm31}, Theroem \ref{compactnessthm} and Theroem \ref{thm24}.
\begin{thm-no}
For $\lambda\in\Lambda^{reg}$ sufficiently large, the Morse-Smale-Witten complex ${\mc C}^\lambda$ of the pair $({\mc F}, g_\lambda)$  is canonically isomorphic to the Morse-Smale-Witten complex of the pair $(f|_{\mu^{-1}(0)}, g|_{\mu^{-1}(0)})$ with grading shifted by one.
\end{thm-no}

\begin{thm-no}
When $\lambda \to 0$, orbits of solutions to (\ref{eqn15a})-(\ref{eqn15b}) converges to certain singular orbits. The counting of isolated singular orbits defines a chain complex ${\mc C}^0$, which is isomorphic to ${\mc C}^\lambda$ for $\lambda\in\Lambda^{reg}$ sufficiently small.
\end{thm-no}

As a corollary, we have
\begin{cor-no} The homology of ${\mc C}^0$ is isomorhpic to the singular homology of $\mu^{-1}(0)$ with grading shifted by one.
\end{cor-no}

Our results give a new perspective on classical Morse theory. If we replace $\mu$ by $\mu-c$ for any regular value $c$ of $\mu$, then the homology of the hypersurface $\mu^{-1}(c)$ can be computed as the homology of the chain complex  ${\mc C}^0$ defined in this paper. The fast flow determined by \eqref{eqn15a}--\eqref{eqn15b} with $\lambda = 0$ is independent of $c$.  Only the slow flow changes as $c$ changes. This allows tracking of the change in homology of $\mu^{-1}(c)$ when $c$ crosses a critical value. However, we do not pursue this point in the present paper. 

Throughout this paper, we shall always assume that $M$, $f$, $\mu$ and $g$ are $C^\infty$ , or {\em smooth}.   A property will be called {\em generic} if  it is true on a countable intersection of open dense sets in a Fr\'{e}chet space of $C^\infty$ functions. 

All chain complexes and homology groups will have ${\mb Z}_2$ coefficients. It should be straightforward to extend our results to integer coefficients. In order to do so, one has to orient the manifold $M$ and various moduli spaces. The isolated orbits will then obtain orientations, and each will be counted with coefficient $1$ or $-1$ in various sums. One then has to keep track of the signs in the adiabatic limit processes. We use ${\mb Z}_2$ coefficients to avoid these complications.

We shall use the notation ${\mb R}_+=(0,\infty)$, ${\mb Z}_+=\{n\in {\mb Z} \ | \ n\ge0\}$, and we use ${\rm index}(p, f)$ to denote the Morse index for $p$ as a critical point of $f$.

\subsection{Outline\label{sec:out}}

 In Section \ref{section2}, we review basic facts about the Morse-Smale-Witten complex, and we construct this complex ${\mc C}^\lambda$ for the function ${\mc F}$ and the metric $g_\lambda$.  Throughout the paper we work with generic data $(f,\mu,g)$.

Let $p_-, \; p_+$ be two critical points of ${\mc F}$, and consider  orbits of (\ref{eqn15a})--(\ref{eqn15b})
from $p_-$ to $p_+$ for different  $\lambda$.  The energy of such an orbit is  ${\mc F}(p_-) - {\mc F}(p_+)$, which is  $\lambda$-independent. 

 For large $\lambda$, because of the finiteness of energy, one can show that all orbits from $p_-$ to $p_+$ will be confined {\it a priori} in a small neighborhood of $\mu^{-1}(0) \times {\mb R}$. In fact, those orbits will be close to the orbits of the negative gradient flow of $f|_{\mu^{-1}(0)}$. This gives the correspondence between the Morse-Smale-Witten chain complex for $({\mc F}, g_\lambda)$ for large $\lambda$, and the complex for $\left( f|_{\mu^{-1}(0)}, g|_{\mu^{-1}(0)} \right)$. In Section \ref{section_infinity} we give two proofs of these facts, one using Fredholm theory and one using geometric singular perturbation theory \cite{Jones_GSPT}. 

On the other hand, in the limit $\lambda \to 0$, one obtains singular orbits of a fast-slow system.  There is  one slow variable, $\eta$.  We will show that for ${\rm index}\, (p_-, {\mc F}) - {\rm index}\, (p_+, {\mc F})=1$,  if $\lambda>0$ is small enough, then near each singular orbit from $p_-$ to $p_+$, there exists a unique orbit of (\ref{eqn15a})--(\ref{eqn15b}) from $p_-$ to $p_+$. This gives a correspondence between the Morse-Smale-Witten chain complex of $({\mc F}, g_\lambda)$ for $\lambda$ small, and a complex ${\mc C}^0$ defined by counting  singular orbits. Details are in Sections \ref{section4} and \ref{section5}.

%In Section \ref{section6}, we sketch how our results give a new approach classical Morse theory. Classical Morse theory is concerned with the change in topology of $\mu^{-1}(c)$ as $c$ varies. We relate this topology to orbits of a fast-slow system. It turns out that the slow manifold is independent of $c$, which gives an approach to tracking these changes.

Our method of studying the $\lambda \to 0$ limit is similar to that used in \cite{Banyaga_Hurtubise}. However, in \cite{Banyaga_Hurtubise}, the slow manifolds are normally hyperbolic, so one can use the classical exchange lemma  \cite{Jones_GSPT} to relate singular orbits for $\lambda=0$ to true orbits for small $\lambda>0$.  In the current paper, on the other hand, one cannot avoid the existence of points on the slow manifold where normal hyperbolicity is lost. To overcome this difficulty one needs a small extension of the more general exchange lemma proved in \cite{Schecter_II}. A review of exchange lemmas including this extension is provided in \ref{app:gspt}.

%\st{Our results can be used to give a new approach to classical Morse theory, which is concerned with the change in topology of $\mu^{-1}(c)$ as $c$ varies.  The slow manifold and fast flow are determined by \eqref{eqn15a}--\eqref{eqn15b} with $\lambda=0$, so they are independent of $c$.  Only the slow flow changes as $c$ changes (replace $\mu^{-1}(0)$ by $\mu^{-1}(c)$ in \eqref{eqn15b}).  This allows tracking of the change in topology of $\mu^{-1}(c)$.   We do not pursue this point in the present paper.}

\subsection{Motivation}

%This work originated from study of symplectic geometry, which is the second named author's primary interest; but the main technique used lies in geometric singular perturbation theory where the first named author has made substantial contributions. This seems to be far-fetched, so we think we owe the reader a long but leisurely account of motivation. 

The motivation for our work comes from quantum physics and symplectic geometry. This subsection is independent of the rest of the paper.

To study quantum physics theories, mathematicians usually consider moduli spaces instead of path integrals.
A moduli space is the space of solutions to some nonlinear differential equation. For example, in Morse theory, which was interpreted by Witten \cite{Witten_Morse} as supersymmetric quantum mechanics, we consider the moduli space of gradient flow lines, which are solutions to the gradient flow equation. Moduli spaces are also used in symplectic geometry.  For example, to define the Gromov-Witten invariant of a symplectic manifold $(X, \omega)$, we consider the moduli space of pseudoholomorphic curves introduced in \cite{Gromov_1985}, which are solutions to a nonlinear Cauchy-Riemann equation on a Riemann surface $\Sigma$. These two examples are both ``conformally invariant,'' that is, the differential equations are independent of the size of the domain, which are the real line ${\mb R}$ in the first example and the Riemann surface $\Sigma$ in the second example. 

We can integrate certain cohomology classes over the moduli spaces. In quantum field theory these integrals are actually correlation functions, which are the most important computable quantities. There are algebraic structures on the set of such integrals. In Morse theory, the integrals can be interpreted as counting the isolated gradient flow lines; they define a chain complex, whose homology is the singular homology of the manifold.  In Gromov-Witten theory, the counting of pseudoholomorphic curves gives rise to the so-called quantum cohomology of the symplectic manifold $(X, \omega)$. 

Our two examples are related by the theory of Floer homology, which was introduced by Floer in 1980s. In Gromov-Witten theory, if the Riemann surface $\Sigma$ is the cylinder ${\mb R} \times S^1$, then we can use a Hamiltonian function $H: X \to {\mb R}$ to perturb the Cauchy-Riemann equation on $\Sigma$. More precisely, if $(t, s)$ is the standard coordinate on the cylinder and $J$ is an almost complex structure on $(X, \omega)$, then Floer's equation is the following PDE for maps $u: {\mb R}\times  S^1  \to X$:
\begin{align}\label{equation18}
{\partial u \over \partial t} + J \left( {\partial u \over \partial s} \right)  + \nabla H (u)  = 0.
\end{align}
This can be viewed as the negative gradient flow equation of a certain function on the loop space of $X$. Moreover, if we shrink loops to points, then solutions should be $s$-independent, so \eqref{equation18} reduces to an equation for maps $u: {\mb R} \to X$:
\begin{align}
{du \over dt} + \nabla H (u) = 0.
\end{align}
This is  the equation considered in the Morse theory. In this sense, in symplectic geometry, Morse theory is often considered to be a finite-dimensional model for the theory of pseudoholomorphic curves.

There are quantum physics theories that are not conformally invariant. Such theories depend on one or more scale parameters. An example is the gauged $\sigma$-model. The corresponding PDE is the symplectic vortex equation, introduced by Mundet in his thesis \cite{Mundet_thesis} and by Cieliebak-Gaio-Salamon \cite{Cieliebak_Gaio_Salamon_2000}. More precisely, suppose we have a symplectic manifold $(\wt{X}, \wt\omega)$ and a Hamiltonian $G$-action, where $G$ is a compact Lie group, with  moment map $\mu: \wt{X} \to {\mf g}$. (This is the motivation for using $\mu$ to denote one of the two functions at the beginning of this paper.) Then the symplectic vortex equation is an elliptic system on a Riemann surface $\Sigma$ associated with the triple $( \wt{X}, \wt\omega, \mu)$. This system depends on a scale parameter $\lambda>0$, which determines the size of $\Sigma$. 

%The above strategy is not always easy to carry out mathematically, though it might be well-understood in physics. Here is an example which is more related to this work. For any compact symplectic manifold $(X, \omega)$, mathematicians defined the so-called ``quantum cohomology'' of $X$, by studying the moduli space of some nonlinear Cauchy-Riemann equation over a Riemann surface. If $X$ is equal to the complex Grassmannian $Gr(k, n)$, then Witten (\cite{Witten_Verlinde}) observed that the quantum cohomology of $X$ is equivalent to another object called the Verlinde algebra. His argument was based on the idea of the variation of scales. Though this has been justified by Agnihotri in his thesis \cite{Agnihotri_thesis} as a mathematical theorem, geometers prefer a proof following Witten's idea. 

One can use a $G$-invariant Hamiltonian function $\wt{H}: \wt{X}  \to {\mb R}$ to perturb the symplectic vortex equation on the cylinder ${\mb R}\times S^1$ and study the corresponding Floer homology theory. This was proposed in \cite{Cieliebak_Gaio_Salamon_2000} and has been studied by the second named author \cite{Xu_VHF}. In this case, the equation is for maps $(u, \Psi): {\mb R}\times S^1 \to \wt{X}\times {\mf g}$, and reads
\begin{align}
\label{equation110} {\partial u \over \partial t} + J \left( {\partial u \over \partial s} \right) + \nabla \left(  \wt{H} + \langle \mu, \Psi\rangle \right) & = 0,\\
\label{equation111} {\partial \Psi \over \partial t} + \lambda^2 \mu(u) & = 0.
\end{align}
Here $\langle \cdot, \cdot \rangle$ is the inner product on the Lie algebra ${\mf g}$, and $\lambda>0$ is the scale parameter.

As shown in \cite{Gaio_Salamon_2005}, the adiabatic limit of the symplectic vortex equation as $\lambda \to \infty$ is nearly the same as the nonlinear Cauchy-Riemann equation that arises in the Gromov-Witten theory of the symplectic quotient $X:= \mu^{-1}(0)/ G$. In the present case, the $\lambda \to \infty$ limit of (\ref{equation110}), (\ref{equation111}) is related to (\ref{equation18}) on the symplectic quotient $X$. 

A natural question arises: what can be said about the opposite adiabatic limit  $\lambda \to 0$? This limit could lead to a different perspective on the quantum cohomology of the symplectic quotient. Moreover, the $\lambda \to 0$ limit is related to the following argument of Witten in \cite{Witten_Verlinde}. The complex Grassmannian $Gr(k, n)$ (the space of $k$-planes in ${\mb C}^n$) can be viewed as the symplectic quotient of the Euclidean space ${\mb C}^{nk}$ with respect to a Hamiltonian $U(k)$-action. Thus as above, we could use the symplectic vortex equation in ${\mb C}^{nk}$ and its $\lambda \to \infty$ adiabatic limit to study the quantum cohomology of $Gr(k, n)$. Witten argued nonrigorously, using path integrals, that the opposite adiabatic limit leads to the Verlinde algebra, discovered by E. Verlinde in \cite{Verlinde_1988}, which is therefore related to the quantum cohomology of $Gr(k, n)$. This argument explains earlier work by Gepner \cite{Gepner}, Vafa \cite{Vafa_92}, and Intriligator \cite{Intriligator}.  

A mathematical proof of the isomorphism between the quantum cohomology and the Verlinde algebra was given by Agnihotri in his thesis \cite{Agnihotri_thesis} by directly calculating the two objects. However, geometers would prefer a proof closer to Witten's argument, using the symplectic vortex equation with varying $\lambda$ and its adiabatic limits. Such a theory should apply not just to Grassmannians but to general symplectic quotients. The difficulty is that, unlike the $\lambda \to \infty$ limit, the $\lambda \to 0$ limit of the symplectic vortex equation is little understood. Only a few isolated results have been obtained, such as \cite{Gonzalez_woodward_small}.
  
  Because of this lack of understanding, it is natural to symplectic geometers to consider a Morse theory model. More precisely, if a solution $(u, \Psi)$ to (\ref{equation110})-(\ref{equation111}) is independent of $s$, then it satisfies the ODE
\begin{align}
\label{equation112} u'(t) + \nabla \left( \wt{H} + \langle \mu, \Psi \rangle \right) & = 0,\\
\label{equation113} \Psi'(t) + \lambda^2 \mu(x) & = 0.
\end{align}
We can further simplify the equation by removing the group action and replacing the moment map $\mu: \wt{X} \to {\mf g}$ by a smooth function $\mu: \wt{X} \to {\mb R}$.\footnote{The second named author would like to thank Urs Frauenfelder for suggesting removal of the group action.} Then (\ref{equation112})-(\ref{equation113}) becomes our equation (\ref{eqn15-1})-(\ref{eqn15-2}), and the symplectic quotient becomes the hypersurface $\mu^{-1}(0)$.  

It may be possible to generalize our result to other situations. For example, we can consider a vector-valued function $\mu: X \to {\mb R}^k$ instead of a scalar-valued one. A similar approach might also apply to Picard-Lefschetz theory, the holomorphic analogue of Morse theory. The case of equation (\ref{equation110})--(\ref{equation111}) can perhaps be attacked in a similar way, presumably using an infinite-dimensional version of the exchange lemma and more-involved functional analytic techniques.

\section{Morse homology and Morse-Smale-Witten complex of $({\mc F}, g_\lambda)$}\label{section2}

\subsection{Morse homology\label{sec:mh}}

For the topological aspects of Morse theory, there is the classical book \cite{Milnor_Morse}. Here we will adopt the viewpoint of Witten \cite{Witten_Morse}, which is by now standard.

Let $M$ be a smooth manifold without boundary. A smooth function $f: M \to {\mb R}$ is a {\em Morse function} if its differential $df$ is a transverse section of the cotangent bundle $T^*M$. Equivalently, at each critical point $p$ of $f$, the  second derivative $D^2f(p)$ is nondegenerate.  By the Morse lemma, near each $p\in {\rm Crit} (f)$, there exists a local coordinate chart $(x_1, \ldots, x_n)$ such that 
\begin{align*}
 f(x)= f(p) -\sum_{i=1}^{k_p} x_i^2 + \sum_{i=k_p+1}^n x_i^2.
\end{align*}
The integer $k_p$ is called the {\em Morse index} of the critical point $p$. We write
\begin{align*}
 {\rm index}\, (p, f) = k_p\in {\mb Z}_+,
\end{align*}
and we denote the set of critical points of $f$ of index $k$ by ${\rm Crit}_{k} (f)$ 

For any complete Riemannian metric $g$ on $M$, the gradient vector field $\nabla f$ is the dual of the 1-form $df$. We denote by $\phi_t: M \to M$ the flow of $-\nabla f$. For each $p\in {\rm Crit} (f)$,    the unstable and stable manifolds are defined by
\begin{align*}
W^u(p) = \left\{ x\in M | \lim_{t\to -\infty} \phi_t(x)= p  \right\}, \quad  W^s(p) = \left\{ x\in M| \lim_{t\to +\infty} \phi_t(x)= p \right\}
\end{align*}
and
\begin{align*}
 {\rm dim}\, W^u(p) = {\rm index}\,  (p, f),\ {\rm dim}\, W^s(p)= n - {\rm index}\, (p, f).
\end{align*}
The pair $(f, g)$ is called {\em Morse-Smale} if for any two critical points $p, q$, $W^u(p)$ and $W^s(q)$ intersect transversally in $M$. This condition condition holds for generic  $(f, g)$ \cite{Schwarz_book}. 
%For $p \neq q$, there is a free ${\mb R}$-action on $W^u(p) \cap W^s(q)$ by time translation along the flow.

Associated to a Morse-Smale pair $(f, g)$ is its {\em Morse-Smale-Witten chain complex} ${\mc C}(f, g)$
of ${\mb Z}_2$-modules.  The module of $k$-chains is generated by the critical points of index $k$.  The boundary of a critical point of index $k$ is a linear combination of critical points of index $k-1$, where the coefficient of a critical point $q$ is the number of orbits of the negative gradient flow from $p$ to $q$.  We write
\begin{align}
 {\mc C}(f, g)= \left( C_*, \partial \right),\ C_k= \bigoplus_{p \in {\rm Crit}_k (f)} {\mb Z}_2 \langle p\rangle,
\end{align}
where $\partial: C_k\to C_{k-1}$ is given by
\begin{align}\label{22}	
 \partial p = \sum_{q\in {\rm Crit}_{k-1} (f)} n_{p, q} \cdot q,\ n_{p, q} = \# \left[(W^u(p)\cap W^s(q))/ {\mb R} \right] \in {\mb Z}_2.
\end{align}
One can show that $\# \left[(W^u(p)\cap W^s(q))/{\mb R} \right]$ is finite and $\partial\circ \partial =0$ (under the Palais-Smale condition on $f$, see \cite{Schwarz_book}), hence ${\mc C}(f, g)$ is a chain complex. Its homology  is called the {\em Morse homology} of $(f, g)$ with ${\mb Z}_2$-coefficients, denoted $H(f, g; {\mb Z}_2)$. If $M$ is compact, this homology is independent of the choice of Morse-Smale pair $(f, g)$ and is isomorphic to the singular homology of $M$. 

\subsection{Morse homology in the analytic setting\label{sec:anal}}

The book \cite{Schwarz_book} give a comprehensive treatment of the analytical perspective on finite-dimensional Morse theory. Here we give a brief review. 

Let $p$ and $q$ be two critical points of $f$, and consider the nonlinear ODE
\begin{align}\label{eqn23}
x' = - \nabla f(x),
\end{align}
with the boundary conditions  
\begin{align}\label{eqn23bc}
\lim_{t\to -\infty} x(t) = p,  \quad \lim_{t\to +\infty} x(t) = q. 
\end{align}
The space of solutions to this boundary value problem 
can be identified with $W^u(p)\cap W^s(q)$ by identifying the solution $x(t)$ with the point $x(0)$.

Indeed, the space of solutions to such a boundary value problem can be viewed as a finite-dimensional submanifold of an infinite-dimensional Banach manifold. We consider the Banach manifold ${\mc B}$ of those $W^{1, 2}_{loc}$ maps $x: {\mb R}\to M$ for which there exists $T>0$ such that for $t\in [T, +\infty)$ (respectively $t\in (-\infty, -T]$), $d(x(t), q)$ (respectively $d(x(t), p)$) is less than the injectivity radius of $(M, g)$ and $\exp_q^{-1}(x(t)) \in W^{1,2}\left( [T, +\infty), T_qM \right)$ (respectively $\exp_p^{-1}(x(t)) \in W^{1, 2}\left( (-\infty, -T], T_p M \right)$). There is a Banach space bundle ${\mc E}\to {\mc B}$ whose fibre over $x: {\mb R}\to M$ is $L^2\left( {\mb R}, x^*TM \right)$. There is a smooth section
\begin{align*}
\begin{array}{cccc}
{\mc S}: & {\mc B} & \to &  {\mc E}\\
         &  x & \mapsto  & x' + \nabla f(x)\end{array}
\end{align*}
whose zero locus is exactly $\wt{\mc M}(p, q)\simeq W^u(p) \cap W^s(q)$, the moduli space of solutions to (\ref{eqn23})--(\ref{eqn23bc}). ${\mc S}$ is a Fredholm section, which means that the linearization of ${\mc S}$ at each $x\in {\mc S}^{-1}(0)$,  given by
\begin{align*}
\begin{array}{cccc}
D{\mc S}_x: & T_x {\mc B} & \to & {\mc E}_x\\
           & V & \mapsto & \nabla_t V + \nabla_V (\nabla f) 
          \end{array}
\end{align*}
is a Fredholm operator.  The Fredholm index is ${\rm index}\, (p) - {\rm index}\, (q)$.

$\wt{\mc M}(p, q)$ is a smooth manifold if  the linearization of ${\mc S}$ is surjective along ${\mc S}^{-1}(0)$. 
%By the} Sard-Smale theorem, this can be achieved by perturbing the pair $(f, g)$ generically. 
In this case, we say that $\wt{\mc M}(p, q)= {\mc S}^{-1}(0)$ is {\em transverse}. This condition is equivalent to transversality of $W^u(p)$ and $W^s(q)$.

$\wt{\mc M}(p, q)$  has a free ${\mb R}$-action by time translation, so we can define ${\mc M}(p, q) = \wt{\mc M}(p, q)/ {\mb R}$, the moduli space of orbits from $p$ to $q$.  A useful assumption is the Palais-Smale condition on $f$ (which is automatic when $M$ is compact): any sequence $x_i \in M$ for which $f(x_i)$ is uniformly bounded and $\left| \nabla f(x_i) \right| \to 0$ has a convergent subsequence. Under the transversality and Palais-Smale assumptions,  if ${\rm index} (q, f) = {\rm index}\,  (p, f) -1$, then ${\mc M}(p, q)$ is finite.  Moreover, if ${\rm index}\, (q, f) = {\rm index}\, (p, f) -2$, then ${\mc M}(p, q)$ can be compactified to become a smooth one-dimensional manifold with boundary; the boundary points correspond to the broken orbits from $p$ to $q$ that pass through another critical point $r$ with ${\rm index}\, (r, f) = {\rm index}\, (p, f) -1$. These two facts imply that the boundary operator (\ref{22}) is well-defined and $\partial\circ \partial = 0$.

A solution of \eqref{eqn23}--\eqref{eqn23bc} has {\em energy}
\begin{align} \label{energydef}
E(x) =  \int_{\mb R} \|x^\prime(t) \|^2 dt  =
- \int_{\mb R} \left\langle \nabla f( x(t)),x^\prime(t)\right\rangle dt =
f(p) - f(q),
\end{align}
which only depends on the values of $f$ at $p$ and $q$. The energy is essential to many useful estimates.

\subsection{The Morse-Smale-Witten complex of $({\mc F}, g_\lambda)$\label{sec:msw}}

We shall assume
\begin{enumerate}
\item[(A1)] $M$ is a compact manifold with metric $g$.
\end{enumerate}

We shall work with generic triples $(f,\mu,g)$ defined on $M$.  For the sake of precision, we shall list our generic assumptions, and denote them (A$k$), $k\ge2$.  In each of our results, we shall assume without comment that all assumptions (A$k$) made up to that point hold.  The assumptions (A$k$) are all independent.

We assume:
\begin{enumerate}
\item[(A2)] $f$ and $\mu$ are Morse functions on $M$.
\item[(A3)] $0$ is a regular value of $\mu$ and $f|_{\mu^{-1}(0)}$ is Morse.
\end{enumerate}
We shall regard the metric $g$  and the Morse function $\mu$ having $0$ as a regular value as arbitrary and fixed.  Then all our assumptions (A$k$), $k\ge2$, are true for generic $f$. Alternatively, we could fix the Morse function $f$; the assumptions are then true for generic $(\mu,g)$.

\begin{lem}\label{lemma21}\cite[Proposition A.2]{Frauenfelder_vortex}
The Lagrange multiplier ${\mc F}$ is a Morse function on $M \times {\mb R}$ and for any $p= (x_p, \eta_p) \in {\rm Crit}({\mc F})$, ${\rm index}(p, {\mc F}) = {\rm index}(x_p, f|_{\mu^{-1}(0)}) + 1$. 
\end{lem}

%\begin{proof}
%Let $p \in {\rm Crit} ({\mc F})$. Since $0$ is a regular value of $\mu$, there exists a local coordinate chart $(x_1, \ldots, x_n)$ near $x_p$ such that $\mu(x_1, \ldots, x_n) = x_n$ and $p$ has coordinates $(0, \ldots, 0)$. Then we see that the second derivative of ${\mc F}$ at $p$ with respect to the coordinates $(x_1, \ldots, x_n, \eta)$ is of the form
%\begin{align}
%D^2 {\mc F}(0) = \left( \begin{array}{ccc} A & b & 0 \\
%                                          b^T & c & 1\\
%                                           0   & 1 & 0 \end{array} \right).
%\end{align}
%Here $A$ is equal to the second derivative of $f|_{\mu^{-1}(0)}$ at $x_p$ with respect to the coordinates $(x_1, \ldots, x_{n-1})|_{\mu^{-1}(0)}$. So this implies that $D^2{\mc F}$ is nondegenerate at $p$, and its determinant is equal to $- \det A$.
%\end{proof}

Now we consider the unstable and stable manifolds of critical points of ${\mc F}$. Since we are dealing with the family of metrics $g_\lambda$, we want transversality of unstable and stable manifolds not just for a single $g_\lambda$, but as much as possible for the family.  Let $\phi_t^\lambda$ denote the flow of (\ref{eqn15a})--(\ref{eqn15b}) for the given value of $\lambda$, let  $p_-, p_+\in {\rm Crit} ({\mc F})$, and let $I$ be an interval in $\lambda$-space.   Define the sets
\begin{align*}
W^u(p_-,\lambda) &= \left\{ p \in M\times {\mb R} \; | \lim_{t\to -\infty} \phi^\lambda_t(p)= p_-  \right\}, \\  
W^s(p_+,\lambda) &= \left\{ p \in (M\times {\mb R} \; | \lim_{t\to +\infty} \phi^\lambda_t(p)= p_+ \right\}, \\
W^u(p_-,I ) &= \left\{ (p,\lambda) \in (M\times {\mb R})  \times I \; | \; p \in  W^u(p_-,\lambda) \right\}, \\  
W^s(p_+,I) &= \left\{ (p,\lambda) \in (M\times {\mb R}) \times I \; | \; p \in  W^s(p_+,\lambda) \right\};
\end{align*}
$$ \wt{\mc M}^I(p_-, p_+) = W^u(p_-, I) \cap W^s(p_+, I),\ {\mc M}^I(p_-, p_+) = \wt{\mc M}^I(p_-, p_+)/ {\mb R}.$$
%For any compact interval $I=[\lambda_1, \lambda_2] \subset {\mb R}_+$, and $p_-, p_+\in {\rm Crit} ({\mc F})$, the manifolds $W^u(p_-,I )$ and  $W^s(p_+,I)$ are transverse for generic $f$; in addition, the manifolds  $W^u(p_-,\lambda )$ and  $W^s(p_+,\lambda)$ are transverse for all but finitely many $\lambda\in I$. Then, by choosing an exhausting sequence of intervals
%\begin{align}
% I_1\subset I_2 \subset \cdots \subset {\mb R}_+
%\end{align}
%one can show that there exists a subset of all $f$, which is of second catogory in a suitable space of functions and metrics, such that $W^u(p_-,{\mb R}_+ )$ and  $W^s(p_+,{\mb R}_+)$ are transverse, and, in addition, $W^u(p_-,\lambda )$ and  $W^s(p_+,\lambda)$ are transverse for all but discretely many $\lambda$.
We assume:
\begin{enumerate}
\item [(A4)] For all $p_-, p_+\in {\rm Crit} ({\mc F})$,  $W^u(p_-,{\mb R}_+ )$ and  $W^s(p_+,{\mb R}_+)$ are transverse. 
\end{enumerate}
Assumption (A4) implies that for all but discretely many $\lambda \in {\mb R}_+$,  $W^u(p_-,\lambda )$ and  $W^s(p_+,\lambda)$ are transverse. We denote by $\Lambda^{reg} \subset {\mb R}_+$ the subset of $\lambda$'s for which $W^u(p_-,\lambda )$ and  $W^s(p_+,\lambda)$ are transverse for all $p_-, p_+\in {\rm Crit} ({\mc F})$.   

\begin{lem}
\label{lem:ps}
For each $\lambda\in {\mb R}_+$, the function ${\mc F}$ satisfies the Palais-Smale condition with respect to the metric $g_\lambda$, i.e., for any sequence $\wt{x}_i \in M \times {\mb R}$ such that ${\mc F}(\wt{x}_i)$ is bounded and $\left| \nabla {\mc F} (\wt{x}_i) \right|_{g_\lambda} \to 0$, 
%(with respect to the fixed metric $g_1$), 
there exists a convergent subsequence of $\wt{x}_i$.
\end{lem}

\begin{proof}
Suppose $\wt{x}_i = (x_i, \eta_i)$.  Since $M$ is compact, we may assume that $\eta_i \to +\infty$ or $-\infty$. The condition that ${\mc F}(\wt{x}_i)$ is bounded implies that $\mu(x_i) \to 0$; but since $0$ is a regular value of $\mu$, $\nabla {\mc F}(\wt{x}_i) = \nabla f(x_i) + \eta_i \nabla \mu(x_i)$ cannot have arbitrary small norm.
\end{proof}

Now, for any $\lambda \in \Lambda^{reg}$, we can define the associated Morse-Smale-Witten complex of $\left( {\mc F}, g_\lambda \right)$, which is denoted by ${\mc C}^\lambda= {\mc C}( {\mc F}, g_\lambda)$. Its homology is denoted by $H^\lambda$. All ${\mc C}^\lambda$ share the same generators and gradings, but the boundary operator $\partial^\lambda$ may change when $\lambda$ crosses a value in ${\mb R}_+\setminus \Lambda^{reg}$. 

In general, the Morse homology of a pair $(f, g)$ is not independent of $(f,g)$ if the underlying manifold is noncompact; see for example \cite{kang_invariance}. 
%One can also see from the Section \ref{section6} in our case how the homology can change if we vary ${\mc F}$ in a specific way. 
Despite the fact that $M\times {\mb R}$ is noncompact, we will show that $H^\lambda$ is independent of  $\lambda$ for $\lambda \in \Lambda^{reg}$.

\begin{lem}\label{lemma23} For any $L>0$, there exists $K_L>0$ such that for any $p_\pm \in {\rm Crit}({\mc F})$, $W^u(p_-, (0, L]) \cap W^s(p_+, (0, L]) \subset \left( M \times [-K_L, K_L] \right) \times [0, L]$.
\end{lem}

\begin{proof}
If this statement is false, then there exists $p_-, p_+\in {\rm Crit}({\mc F})$, a sequence $\lambda_i$ converging to $\lambda_\infty \in [0, L]$ and $(x_i, \eta_i) \in W^u(p_-, \lambda_i) \cap W^s(p_+, \lambda_i)$ such that $\lim_{i \to \infty} |\eta_i| = + \infty$. Let $\wt{p}_i(t)=(\wt x_i(t),\wt \eta_i(t))$
be the solution of (\ref{eqn15a})--(\ref{eqn15b}) for $\lambda = \lambda_i$ with $\wt{p}_i(0)=(x_i, \eta_i)$. Then since the $\lambda_i$ are bounded, we see that for any $R>0$,  
$| \wt\eta_i(t) |\to \infty$ uniformly on $-R\le t\le R$ as $i\to\infty$.

Now
\begin{align}
{\mc F}(p_-) \geq {\mc F}(\wt p_i(t)) = f(\wt x_i(t)) + \wt\eta_i(t) \mu(\wt x_i(t)) \geq {\mc F}(p_+)
\end{align}
which implies that $\mu(\wt x_i(t))\to 0$ uniformly on $-R\le t\le R$. Hence, since 0 is a regular value of $\mu$, $ \left\| \nabla \mu(\wt x_i(t)) \right\|$ on $-R\le t\le R$ is bounded away from 0 for large $i$.
Therefore, by the definition \eqref{energydef} of the energy of the solution $\wt{p}_i$, we see
\begin{multline*}
{\mc F}(p_+) - {\mc F}(p_-) \geq  \int_{-R}^R \left\| \wt p^\prime_i(t))\right\|^2 dt  \geq  \int_{-R}^R \left\| \wt x^\prime_i(t))\right\|^2 dt
\\= \int_{-R}^R \left\| \nabla f(x_i(t)) + \eta_i(t) \nabla \mu(x_i(t)) \right\|^2 dt \to \infty,
\end{multline*} 
which is impossible.
\end{proof}

\begin{prop}  For any $\lambda_1, \lambda_2\in \Lambda^{reg}$, there is a canonical isomorphism $\Phi_{\lambda_1, \lambda_2}: H^{\lambda_1} \to H^{\lambda_2}$ such that for $\lambda_1, \lambda_2, \lambda_3 \in \Lambda^{reg}$, $\Phi_{\lambda_2, \lambda_3} \circ \Phi_{\lambda_1, \lambda_2} = \Phi_{\lambda_1, \lambda_3}$.
\end{prop}

\begin{proof}
For any $\lambda_1, \lambda_2 \in \Lambda^{reg}$, $\lambda_1< \lambda_2$, we can compare the two complexes ${\mc C}^{\lambda_1}$ and ${\mc C}^{\lambda_2}$ as in the case of compact manifolds, thanks to the compactness provided by the previous lemma. More precisely, we can either use the continuation principle  \cite{Schwarz_book}, or  bifurcation analysis as in \cite{Floer_intersection}, to show that the two chain complexes have isomorphic homology. Note that as we vary $\lambda$ from $\lambda_1$ to $\lambda_2$, the critical point set is fixed so there is no ``birth-death'' of critical points, but there may be a loss of transversality between unstable and stable manifolds at discrete values of $\lambda$.
\end{proof}

\section{Adiabatic limit  $\lambda\to \infty$\label{section_infinity}}

We assume
\begin{enumerate}
\item[(A5)] $(f|_{\mu^{-1}(0)}, g|_{\mu^{-1}(0)} )$ is Morse-Smale.
\end{enumerate}
\begin{thm}\label{thm31}
For $\lambda\in \Lambda^{reg}$ sufficiently large, the complex ${\mc C}^\lambda$ is canonically isomorphic to the Morse-Smale-Witten complex of the pair $(f|_{\mu^{-1}(0)}, g|_{\mu^{-1}(0)})$ with grading shifted by one.
\end{thm}

\begin{cor}\label{cor32}
For any $\lambda \in \Lambda^{reg}$, there is a canonical isomorphism
$$\Phi_{\lambda, +\infty} : H^\lambda_* \to H_{*-1}( \mu^{-1}(0), {\mb Z}_2).$$
\end{cor}

We give two proofs of Theorem \ref{thm31}, one using the infinite-dimensional implicit function theorem and one using geometric singular perturbation theory.  The first is more likely to generalize. The second gives more geometric intuition. There are other ways to relate the two chain complexes. For example, in the appendix of \cite{Frauenfelder_vortex} Frauenfelder had a different approach by deforming the function $f$ in the normal direction of $\mu^{-1}(0)$.

\subsection{Proof of Theorem \ref{thm31} using the implicit function theorem}

We first give a sketch of this proof. For fixed $p_\pm \in {\rm Crit} {\mc F}$, we prove in Proposition \ref{prop33} that, for $\lambda$ sufficiently large, any orbit in ${\mc M}^\lambda(p_-, p_+)$ will be close to some orbit of the negative gradient flow of $f|_{\mu^{-1}(0)}$ from $p_-$ to $p_+$. Such orbits form a moduli space ${\mc N}^\infty(p_-, p_+)$. Next we show that, for large $\lambda$, there exists a homeomorphism $\Phi^\lambda: {\mc N}^\infty(p_-, p_+) \to {\mc M}^\lambda(p_-, p_+)$, which is constructed by using the infinite dimensional implicit function theorem. Proposition \ref{prop33} is used to prove the surjectivity $\Phi^\lambda$. Since the Morse homology is defined by counting orbits connecting two critical points with adjacent Morse indices, the homeomorphism (Theorem \ref{thm36}) means that the counting of ${\mc M}^\lambda(p_-, p_+)$ and that of ${\mc N}^\infty(p_-, p_+)$ are the same. Noting that the indices of $p_\pm$ are dropped by 1 when regarded as critical points of $f|_{\mu^{-1}(0)}$, Theroem \ref{thm31} follows immediately.

\subsubsection{Convergence to orbits in the level set}

For any $\wt{p} = (\wt{x}, \wt\eta) \in \wt{\mc M}^\lambda(p_-, p_+)$, its energy, calculated by the metric $g_\lambda$ on $M \times {\mb R}$, is
\begin{align}
E= {\mc F}(p_-) - {\mc F}(p_+) = \left\| \wt{x}'\right\|_{L^2}^2 + \lambda^2 \left\| \mu( \wt{x}) \right\|_{L^2}^2.
\end{align}
Here the $L^2$-norms are still defined using the fixed metric $g_1$.
So
\begin{align}
\left\| \mu(\wt{x})\right\|_{L^2} \leq {E^{1\over 2} \over \lambda}.
\end{align}
On the other hand, one has
\begin{align}
 \left\| {d\over dt} \mu( \wt{x}(t)) \right\|_{L^2} \leq \left\| d\mu \right\|_{L^\infty} \left\| \wt{x}'(t) \right\|_{L^2}.
\end{align}
So by Sobolev embedding $W^{1, 2} \to C^0$, for any $\epsilon>0$, there exists $\Lambda_\epsilon>0$ such that for $\lambda>\Lambda_\epsilon$ and $\wt{x}\in \wt{\mc M}^\lambda(p_-, p_+)$, we have $\wt{x}(t)\in U_\epsilon = \mu^{-1}((-\epsilon, \epsilon))$.

Now, consider the Banach manifold ${\mc B}= {\mc B}^{1, 2} $ of $W^{1, 2}_{loc}$-maps $\wt{p}= (\wt{x}, \wt\eta)$ from ${\mb R} \to U_\epsilon \times {\mb R}$ such that $\wt{p}$ is assymptotic to $p_\pm= \left( x_\pm, \eta_\pm \right) \in {\rm Crit}({\mc F})$ at $\pm \infty$ in the following sense: there exists $R>0$ and $\wt{W}_- \in W^{1, 2}((-\infty, -R], T_{p_-} M \oplus {\mb R})$ and $\wt{W}_+\in W^{1, 2}([R, +\infty), T_{p_+}M \oplus {\mb R})$ such that $\wt{p}|_{(-\infty, -R]} = \exp_{p_-} \wt{W}_-$, $\wt{p}|_{[R, +\infty)} = \exp_{p_+} \wt{W}_+$. Consider ${\mc E} \to {\mc B}$ the Banach space bundle whose fibre over $\wt{p}$ is ${\mc E}_{\wt{p}}= L^2( \wt{x}^*TM \oplus {\mb R})$. For $\lambda > \Lambda_\epsilon$, consider the Fredholm section $\wt{S}^\lambda: {\mc B}\to {\mc E}$ given by
\begin{align}
 \wt{S}^\lambda(\wt{p}) = \wt{S}^\lambda( \wt{x}, \wt\eta) = \left( {d\wt{x}\over dt} + \nabla f + \wt\eta \nabla \mu, {d\wt\eta\over dt} + \lambda^2 \mu(\wt{x}) \right). 
\end{align}
Then $\wt{\mc M}^\lambda(p_-, p_+) = \left( \wt{S}^\lambda \right)^{-1}(0)$, where $0$ is the zero section of ${\mc E}\to {\mc B}$. The linearization of $\wt{\mc S}^\lambda$ at $\wt{p}=(\wt{x}, \wt\eta)\in {\mc B}$ is given by 
\begin{align}
\begin{array}{cccc}
\wt{\mc D}_{\wt{p}} : & T_{\wt{p}} {\mc B} & \to & {\mc E}_{\wt{p}}\\
                      & \left( \begin{array}{c}  V \\ H \end{array} \right) & \mapsto &  \left( \begin{array}{c} \nabla_t V + \nabla_V ( \nabla f + \eta \nabla \mu) + H \nabla \mu \\
                       {dH\over dt} + \lambda^2 d\mu(V)  \end{array}\right)\end{array}
\end{align}

On $U_\epsilon$, there is the line bundle $L\subset TM$ generated by $\nabla \mu$, and denote by $L^\bot$ the orthogonal complement with respect to the Riemannian metric $g$. Then, for any $\wt{p}\in {\mc B}$, we can decompose the domain and target space of $\wt{D}_{\wt{p}}$ as
\begin{align}\label{517}
 T_{\wt{p}}{\mc B} \simeq W_L(\wt{p}) \oplus W_T( \wt{p}),\  W_L(\wt{p}) = W^{1, 2}( \wt{x}^*L \oplus {\mb R}),\  W_T( \wt{p}) = W^{1, 2} ( \wt{x}^* L^\bot );
\end{align}
\begin{align}\label{518}
 {\mc E}_{\wt{p}} \simeq {\mc E}_L( \wt{p}) \oplus {\mc E}_T( \wt{p}),\  {\mc E}_L( \wt{p}) = L^2 ( \wt{x}^*L \oplus {\mb R}),\ {\mc E}_T( \wt{p}) = L^2( \wt{x}^* L^\bot ).
\end{align}
We rescale the norms on $W_L( \wt{p})$ and ${\mc E}_L(\wt{p})$ as follows. We identify $(h_1 \nabla \mu, h_2) \in W_L(\wt{p})$ with $(h_1, h_2) \in W^{1, 2}\oplus W^{1, 2}$ and define
\begin{align}\label{eqn48}
 \left\| (h_1, h_2 ) \right\|_{W_\lambda} = \lambda \left\| h_1  \right\|_{L_2} + \left\| h_1'  \right\|_{L^2} + \left\| h_2 \right\|_{L^2} + \lambda^{-1} \left\| h_2' \right\|_{L^2};
\end{align}
and for $(h_1 \nabla \mu, h_2) \in {\mc E}_L( \wt{p})$, identify it with $(h_1, h_2) \in L^2\oplus L^2$ and define
\begin{align}\label{eqn49}
\left\| ( h_1 , h_2 ) \right\|_{L_\lambda} = \left\| h_1  \right\|_{L^2} + \lambda^{-1} \left\| h_2 \right\|_{L^2}.
\end{align}
We leave the norms on their complements unchanged, and use $W_\lambda$ and $L_\lambda$ to denote the norms on $T_{\wt{p}} {\mc B}$ and ${\mc E}_{\wt{p}}$ respectively. 

Now we describe the limit objects. There exists a smooth function $\zeta: U_\epsilon \to {\mb R}$ defined by the condition
\begin{align}
 \langle \nabla \mu (x) , \nabla f (x) + \zeta (x) \nabla \mu(x) \rangle =0.
\end{align}
Then $\nabla f + \zeta \nabla \mu$ is a smooth vector field whose restriction to $\mu^{-1}(0)$ is the gradient of $\ov{f}= f|_{\mu^{-1}(0)}$ with respect to the restriction of the Riemannian metric. We denote by $\wt{\mc N}^\infty( p_-, p_+)$ the space of solutions of the negative gradient flow of $\ov{f}$, whose elements are denoted by $y: {\mb R} \to \mu^{-1}(0)$; and by ${\mc N}^\infty(p_-, p_+)$ the quotient space by identifying reparametrizations, and by $\ov{\mc N}^\infty( p_-, p_+)$ the compactified moduli space by adding broken orbits.

\begin{prop}\label{prop33} Suppose $\lambda_\nu \to \infty$ and $\wt{p}_\nu= (\wt{x}_\nu, \wt\eta_\nu) \in \wt{\mc M}^{\lambda_\nu}(p_-, p_+)$. Then there is a subsequence, still indexed by $\nu$, and a broken orbit
${\mc Y}= \left(\left[ y_i \right]  \right)_{i=1}^n \in \ov{\mc N}^\infty(p_-, p_+)$ (where $[y_i]$ is the orbit of $y_i$) and a constant $c_0>0$ such that
\begin{enumerate}
\item $\left\| \left( \mu(\wt{x}_\nu), \wt\eta_\nu- \zeta( \wt{x}_\nu) \right) \right\|_{W_{\lambda_\nu}}\leq c_0 \lambda_\nu^{-1}$;

\item There exists $t_{1, \nu}, t_{2, \nu}, \ldots, t_{n, \nu}\in {\mb R}$ such that $\wt{x}_\nu(t_{i, \nu}+ \cdot)$ converges to $y_i$ in $C^1_{loc}$-topology;
\end{enumerate}
\end{prop}

\begin{proof} 
Apply $d\mu$ to the equation $\wt{x}_\nu'(t) + \nabla f( \wt{x}_\nu(t)) + \wt\eta_\nu(t) \nabla \mu ( \wt{x}_\nu(t)) = 0$, we obtain
\begin{multline}\label{eqn411}
0= {d\over dt} \mu( \wt{x}_\nu(t)) + \langle \nabla \mu ( \wt{x}_\nu(t)), \nabla f (\wt{x}_\nu(t)) + \wt\eta_\nu(t) \nabla \mu( \wt{x}_\nu(t))\rangle \\
= {d\over dt} \mu( \wt{x}_\nu(t)) + \lambda_\nu \left| \nabla \mu \right|^2 \left( {1\over \lambda_\nu} ( \wt\eta_\nu(t)- \zeta( \wt{x}_\nu(t))) \right) .
\end{multline}
Also we have 
\begin{align}\label{eqn412}
 {d\over dt} \left( {1\over \lambda_\nu} \left( \wt\eta_\nu- \zeta ( \wt{x}_\nu)  \right)  \right) + \lambda \mu( \wt{x}_\nu) = - {1\over \lambda_\nu} {d\over dt} \zeta (\wt{x}_\nu ).
\end{align}
Consider the linear operator
\begin{align}
\begin{array}{cccc}
 D: & W^{1, 2}({\mb R}, {\mb R}^2) & \to &  L^2( {\mb R}, {\mb R}^2)\\
 & (f_1, f_2)  & \mapsto & \left( {d\over dt} f_1 + \lambda \left|\nabla\mu\right|^2 f_2, {d\over dt } f_2 + \lambda  f_1 \right)
\end{array}.
\end{align}
If regarded as an unbounded operator from $L^2$ to $L^2$, then it is bounded from below by $c \lambda$ for some constant $c$. Then (\ref{eqn411}) and (\ref{eqn412}) imply that
\begin{align}
 \left\| \mu( \wt{x}_\nu) \right\|_{L^2} + \lambda_\nu^{-1} \left\| \wt\eta_\nu - \zeta (x_\nu) \right\|_{L^2} \leq {c\over \lambda_\nu^2} \left\| {d\over dt} \zeta( \wt{x}_\nu) \right\|_{L^2}\leq c \lambda_\nu^{-2}.
\end{align}
Also, $D$ has a uniformly bounded right inverse, so
\begin{align}
\left\| \mu( \wt{x}_\nu) \right\|_{W^{1, 2}} + \lambda_\nu^{-1}\left\| \wt\eta_\nu- \zeta( \wt{x}_\nu) \right\|_{W^{1,2}}\leq c \lambda_\nu^{-1}.
\end{align}
This implies the first claim of this proposition. Moreover, the Sobolev embedding $W^{1, 2}\to C^0$ implies in particular that $\wt\eta_\nu$ is uniformly bounded since $\zeta(x_\nu)$ is. Then by the differential equation on $\wt{x}_\nu$, we have that $\left| \wt{x}_\nu' \right|$ is uniformly bounded.

Now we identify $U_\epsilon \simeq \mu^{-1}(0)\times (-\epsilon, \epsilon)$ such that the projection to the second component is equal to $\mu$. Then we can write
\begin{align}
\wt{x}_\nu(t)= \left( \ov{x}_\nu(t), \mu( \wt{x}_\nu(t))\right).
\end{align}
Projecting to the first factor and using the fact that $\left| \wt{x}_\nu' \right|$ is uniformly bounded, we see that there exists $K>0$ independent of $\nu$, such that 
\begin{align}
\left| \ov{x}_\nu'(t) + \nabla \ov{f}\left( (\ov{x}_\nu(t)) \right) \right|\leq  K | \mu( \wt{x}_\nu(t))|.
\end{align}
This implies that a subsequence of $\ov{x}_\nu$ converges to a broken orbit for the induced function $\ov{f}$ in $C^0_{loc}$-topology. More precisely, there exists $t_{1, \nu}, \ldots, t_{k, \nu}\in {\mb R}$ such that $\ov{x}_\nu(t_{i, \nu} + \cdot) $ converges to an orbit $y_i$ in $\mu^{-1}(0)$ in $C^0_{loc}$-topology.

Then, by the Sobolev embedding $W^{1, 2} \to C^0$ and the $C^0_{loc}$ convergence of $ \wt{x}_\nu(t_{i, \nu}+ \cdot)$ to $y_i$, we see that $\wt\eta_\nu( t_{i, \nu}+ \cdot)$ converges to $\zeta(y_i)$ in $C^0_{loc}$. This implies that $ \wt{x}_\nu(t_{i, \nu} + \cdot)$ converges to $y_i$ in $C^1_{loc}$. 
\end{proof}

\subsubsection{Applying implicit function theorem and the isomorphism of chain complexes}

Now we want to prove that any $y \in \wt{\mc N}^\infty(p_-, p_+)$ can be approximated by $\wt{p} \in \wt{\mc M}^\lambda(p_-, p_+)$, so that in particular, when ${\rm index} (p_-, {\mc F}) - {\rm index}\,  (p_+, {\mc F}) = 1$, there is a canonical one-to-one correspondence between ${\mc N}^\infty(p_-, p_+)$ and ${\mc M}^\lambda(p_-, p_+)$.

Since we have assumed that the restriction of $(f, g)$ to $\mu^{-1}(0)$ is Morse-Smale, the trajectory $y$ is transverse. Namely, the linearized operator
\begin{align}
\begin{array}{cccc}
 \ov{D}_y: & W^{1, 2}( y^* T\mu^{-1}(0) ) & \to & L^2( y^* T\mu^{-1}(0))\\
           & V & \mapsto & \nabla_{y'} V + \nabla_V ( \nabla f + \zeta \nabla \mu) 
           \end{array}
\end{align}
is surjective and has a bounded right inverse $\ov{Q}_y$. In particular, we can choose $\ov{Q}_y$ such that 
\begin{align}
 {\rm Im} \, \ov{Q}_y= \left\{ V\in W^{1, 2}( y^*T \mu^{-1}(0) )\ |\ g(V(0), W(0))=0 \ \forall W\in {\rm ker} \ov{D}_y\right\}.
\end{align}

For any large $\lambda$, we take our approximate solution just to be 
\begin{align}
 \wt{y}(t) = \left( y(t), \zeta(y(t)) \right).
\end{align}
Note that $y(t)$ converges to $p_\pm$ exponentially as $t\to \pm\infty$, so $\wt{y} \in {\mc B}$. Then denote the linearization of $\wt{S}^\lambda$ at $\wt{y}$ by
\begin{align}
\begin{array}{cccc}
 \wt{D}_{\wt{y}}: & W^{1, 2}( y^*TM \oplus {\mb R}) &  \to &  L^2 (y^*TM \oplus {\mb R})\\
                  & \left(\begin{array}{c} V\\ h      \end{array} \right) & \mapsto & \left(\begin{array}{c}  \nabla_{y'}	 V + \nabla_V ( \nabla f + \zeta \nabla \mu) + h \nabla \mu\\
                  h' + \lambda^2 d\mu(V) \end{array}\right).
                  \end{array}
\end{align}

Now we construct a right inverse to $\wt{D}_{\wt{y}}$ out of $\ov{Q}_y$. Note that, for any $\ov{V}\in W_T(\wt{y})$
\begin{align}
 \wt{D}_{\wt{y}} (\ov{V}) =  \left( \nabla_t \ov{V} + \nabla_{\ov{V}} ( \nabla f + \zeta \nabla \mu ), \lambda^2 d\mu( \ov{V}) \right) = \left( \ov{D}_y (\ov{V}), 0 \right) \in {\mc E}_T( \wt{y}).
\end{align}
Hence with respect to the decomposition (\ref{517}) and (\ref{518}), the linearized operator can be written as
\begin{align}\label{equation323}
 \wt{D}_{\wt{y}} = \left( \begin{array}{cc}  \ov{D}_y & A_y \\ 0 & D'_{\wt{y}}    \end{array}\right).
\end{align}
Here $A_y$ is given by
\begin{align}
 A_y( h_1 \nabla \mu, h_2)  = h_1 W
\end{align}
where $W$ is a smooth tangent vector field on $\mu^{-1}(0)$. Hence there exists $c_1>0$ such that
\begin{align}
\left\| A_y(h_1 \nabla \mu, h_2 ) \right\|_{L_\lambda} \leq c_1 \left\| h_1 \right\|_{L^2} \leq c_1 \lambda^{-1} \left\| h_1 \right\|_{W_\lambda}.
\end{align}

Now we look at the operator $D'_{\wt{y}}$ in (\ref{equation323}). After trivialize the bundle ${\mb R}\{ \nabla \mu \} \oplus {\mb R}$ isometrically to ${\mb R}^2$, we see the operator $D'_{\wt{y}}$ is transformed into
\begin{align}
 {\mf D}'_{\wt{y}}= {d\over dt} + \left( \begin{array}{cc} a(t) & \lambda b(t) \\
                                                            \lambda b(t) & 0    \end{array}          \right) := {d\over dt} + B(t).
\end{align}
It is easy to see that, for $\lambda$ large, $B(t)$ has 1 positive eigenvalue and one negative eigenvalue, both of which are bounded away from zero by $c \lambda$, where $c$ is a constant independent of $\lambda$. Hence we have a (unique)bounded right inverse 
\begin{align}
 Q'_{\wt{y}}: {\mc E}_L( \wt{y}) \to W_L(\wt{y}),\  \left\| Q'_{\wt{y}} \right\|\leq c_2
\end{align}
for some $c_2>0$. Then define
\begin{align}
 Q_{\wt{y} } = \left( \begin{array}{cc} \ov{Q}_y & 0 \\ 0 & Q'_{\wt{y}} \end{array} \right)
\end{align}
which serves as the approximate right inverse. 

\begin{lem}
 There exists $\Lambda_0>0$, such that for all $\lambda > \Lambda_0$,
\begin{align}
 \left\| {\rm Id} - \wt{D}_{\wt{y}} Q_{\wt{y}} \right\|_{L_\lambda} < {1\over 2}
\end{align}
with respect to the operator norm of the space $L_\lambda$.
\end{lem}

\begin{proof}
Note that ${\rm Id} - \wt{D}_{\wt{y}} Q_{\wt{y}} = A_y \circ Q_{\wt{y}}'$ and 
\begin{align}
 \left\| A_{\wt{y}} Q'_{\wt{y}}(h_1 \nabla \mu, h_2) \right\|_{L_\lambda} \leq {c_1\over \lambda}  \left\| Q'_{\wt{y}} (h_1 \nabla \mu, h_2) \right\|_{W_\lambda} \leq {c_1 c_2 \over \lambda} \left\| (h_1 \nabla \mu, h_2) \right\|_{L_\lambda}.
\end{align}
\end{proof}

Hence $\wt{D}_{\wt{y}} Q_{\wt{y}}$ is invertible and a right inverse can be constructed as
\begin{align}
\wt{Q}_{\wt{y}} = Q_{\wt{y}} \left( \wt{D}_{\wt{y}} Q_{\wt{y}} \right)^{-1}: L_\lambda \to W_\lambda.
\end{align}
It is uniformly bounded by some constant which is independent of $\lambda$.

Now since we want to apply the implicit function theorem (cf. \cite[Appendix A]{McDuff_Salamon_2004}), we need to trivialize the Banach manifold ${\mc B}$ and the Banach space bundle ${\mc E}$ locally near $\wt{y}$. We identify a neighborhood of $\wt{y}$ in ${\mc B}$ with a small ball (with respect to the $W_\lambda$-norm) centered at the origin of $T_{\wt{y}} {\mc B}$ as follows: using the identification $U_\epsilon \simeq \mu^{-1}(0) \times (-\epsilon, \epsilon)$, for any $\wt{W} = (W, a \nabla \mu, h) \in T_{\wt{y}} {\mc B}$, we identify it with the map $\rho(\wt{W})= (\ov{\exp}_{\wt{y}} W, a, \zeta(y) + h )$ into $\mu^{-1}(0)\times (-\epsilon, \epsilon) \times {\mb R} $ where $\ov{\exp}$ is the exponential map inside $\mu^{-1}(0)$. We trivialize ${\mc E}$ over such a neighborhood around $\wt{y}$ by parallel transport along the radial  geodesics, which is denoted by $\Phi_{\wt{W}}: {\mc E}_{\rho(\wt{W})} \to {\mc E}_{\wt{y}}$. Then for $\wt{W}\in T_{\wt{y}} {\mc B}$ with $\left\| \wt{W}\right\|_{W_\lambda}$ small, denote by $\wt{D}_{\wt{W}}: T_{\wt{y}} {\mc B}\to {\mc E}_{\wt{y}}$ the linearization of $ \Phi_{\wt{W}} \circ \wt{S}^\lambda( \rho(\wt{W}))$ at $\wt{W}$. We have
\begin{lem}
There exists $\epsilon_3, c_3>0$, independent of $\lambda$, such that for all $\wt{W}\in T_{\wt{y}}{\mc B}$ with $\left\| \wt{W} \right\|_{W_\lambda} < \epsilon_3$, one has
\begin{align}\label{eqn432}
 \left\| \wt{D}_{\wt{W}} - \wt{D}_{\wt{y}} \right\| \leq c_3 \left\| \wt{W}\right\|_{W_\lambda}.
\end{align}
\end{lem}

\begin{proof}
First consider the case when the ${\mb R}$-component of $\wt{W}$ is zero. We denote by $\wt{S}_1$ the first component of $\wt{S}^\lambda$, which is independent of $\lambda$. Then we see 
\begin{align}\label{eqn433}
\left\| D\wt{S}_1(\wt{W}) - D\wt{S}_1(\wt{y})\right\| \leq c \left\| \wt{W}\right\|_{L^\infty}
\end{align}
which is standard. Now because the ${\mb R}$-component of $\wt{W}$ is zero, we have $\left\| \wt{W}\right\|_{L^\infty} \leq c \left\| \wt{W}\right\|_{W_\lambda}$ by our definition of the $W_\lambda$-norm (\ref{eqn48}).

The second component of $\wt{S}^\lambda$ is denoted by $\wt{S}^\lambda_2$. Then we see in this case, for any $\wt{V}= (V, h) \in T_{\wt{y}} {\mc B}$
\begin{align}\label{eqn434}
D\wt{S}^\lambda_2( \wt{W}) (\wt{V}) =\lambda^2 h. 
\end{align}
Hence the variation of the derivative of $\wt{S}^\lambda_2$ is always zero.

Now we consider the case $\wt{W}= (0, h_0)$ with $h_0\in W^{1, 2}({\mb R})$. Then for any $\wt{W}_0 \in T_{\wt{y}} {\mc B}$ small, by our definition of our norms (\ref{eqn48}) and (\ref{eqn49}),
\begin{multline}\label{eqn435}
\left\| (\wt{D}_{\wt{W} + \wt{W}_0} - \wt{D}_{\wt{W}_0})(V, h)\right\|_{L_\lambda}  = \left\| h_0 \nabla_V \nabla \mu \right\|_{L_\lambda} = \left\| h_0 \nabla_V \nabla \mu\right\|_{L^2} \\
\leq c \left\| \wt{W}\right\|_{W_\lambda} \left\| V\right\|_{L^2}  \leq c \left\| \wt{W}\right\|_{W_\lambda} \left\| (V, h) \right\|_{W_\lambda}
 \end{multline}
for some $c>0$. 

Combining (\ref{eqn433}), (\ref{eqn434}) and (\ref{eqn435}) we obtain (\ref{eqn432}).
\end{proof}

On the other hand, we see
\begin{align}
 \left\| \wt{S}^\lambda( \wt{y}) \right\|_{L_\lambda} = \lambda^{-1} \left\| {d\over dt} \zeta(y)\right\|_{L^2} \leq c_4 \lambda^{-1}
\end{align}
for some constant $c_4>0$. By the implicit function theorem, there exists $\epsilon_4>0$ such that for sufficiently large $\lambda$, there exists a unique $\wt{W}_\lambda(y) \in T_{\wt{y}} {\mc B}$ satisfying
\begin{align}
 \left\| \wt{W}_\lambda(y) \right\|_{W_\lambda} \leq \epsilon_4,\ \wt{W}_\lambda(y) \in {\rm Im} \, Q_{\wt{y}},\ \wt{S}^\lambda( \exp_{\wt{y}} \wt{W}_\lambda(y) ) =0.
\end{align}
Moreover, there exists $c_5>0$ such that
\begin{align}
 \left\| \wt{W}_\lambda(y) \right\|_{W_\lambda} \leq {c_5 \over \lambda}.
\end{align}
We hence define the {\it gluing map} (for large $\Lambda_0$) to be
\begin{align}
\begin{array}{cccc}
 \Phi: & [\Lambda_0, +\infty) \times \wt{\mc N}^\infty(p_-, p_+)  & \to  &  \cup_{\lambda \geq \Lambda_0} \wt{\mc M}^\lambda( p_-, p_+)\\
       &   \left( \lambda, y \right) & \mapsto & \exp_{\wt{y}} \wt{W}_\lambda(y).
\end{array}
\end{align}

\begin{thm}\label{thm36} For each $\lambda \in [\Lambda_0, +\infty) \cap \Lambda^{reg}$, the restriction of $\Phi$ to $\{ \lambda \} \times \wt{\mc N}^\infty (p_-, p_+)$ is a homeomorphism onto $\wt{\mc M}^\lambda(p_-, p_+)$.
\end{thm}

\begin{proof}
To complete the proof of this theorem, it remains to prove the local surjectivity of the gluing map $\Phi$. For simplicity, we only prove this for the case in which the Morse indices of $p_-$ and $p_+$ differ by one, which is enough for our purpose. If, in this case, the gluing map is not locally surjective, then there exists a sequence $\lambda_i \to \infty$ and $\wt{p}_i \in \wt{\mc M}^{\lambda_i} (p_-, p_+)$ which converges to some $\wt{y}\in \wt{\mc N}^\infty(p_-, p_+)$ as described by Proposition \ref{prop33} (with $n =1$ and $t_{1, \nu} = 0$) but don't lie in the range of $\Phi$. Consider the hyperplanes $N(y(s)) = (y'(s))^{\bot} \subset T_{y(s)} M$ for $s\in {\mb R}$. For each $i$, there exists $s_i\in {\mb R}$ such that $s_i \to 0$ and  
\begin{align}
 \wt{p}_i(0) \in \left\{(\exp_{y(s_i)} V, \eta) \ |\  \eta\in {\mb R},\ V\in N(y( s_i )),\ \|V\|\leq \epsilon \right\}.
\end{align}
Then if we write $\wt{p}_i(t) = \exp_{\wt{y}(s_i + t)} \wt{W}_i(t) $, then $\wt{W}_i \in {\rm Im} \, Q_{\wt{y}( s_i + \cdot)}$ and $\left\| \wt{W}_i\right\|_{W_{\lambda_i}} \leq c_0 \lambda_i^{-1}$ by Proposition \ref{prop33}. By the uniqueness part of the implicit function theorem, this implies that $\wt{p}_i = \Phi( \lambda_i, \wt{y}(s_i + \cdot) )$, which contradicts our assumption.
\end{proof}

Now Theorem \ref{thm31} is an immediate consequence. 

\subsection{Proof of Theorem \ref{thm31} using geometric singular perturbation theory\label{sec:gsptpf}}

In the system (\ref{eqn15a})--(\ref{eqn15b}), in which the time variable is $t$, we make the rescalings
$$
\eta = \frac{\rho}{\epsilon}, \quad \lambda = \frac{1}{\epsilon}, \quad t=\epsilon\tau.
$$
The system becomes
\begin{align}
\frac{dx}{d\tau} &= -\left(\epsilon\nabla f(x) + \rho \nabla \mu(x)\right), \label{s11} \\
\frac{d\rho}{d\tau}&= -\mu(x). \label{s12}
\end{align}
Studying the limit $\lambda\to\infty$ in (\ref{eqn15a})--(\ref{eqn15b}) is equivalent to studying the limit $\epsilon\to0$ in  (\ref{s11})--(\ref{s12}).
Setting $\epsilon=0$ in  (\ref{s11})--(\ref{s12}), we obtain
\begin{align}
\frac{dx}{d\tau} &= - \rho \nabla \mu(x), \label{s21} \\
\frac{d\rho}{d\tau}&= -\mu(x). \label{s22}
\end{align}
Let $N$ denote the submanifold of $M$ defined by $\mu=0$.  Then
the set of equilibria of  (\ref{s21})--(\ref{s22}) is the set $\mu=\rho=0$, i.e., $N_0=N \times \{0\} \subset M \times {\mb R}$.  $N_0$  is a  compact codimension-two submanifold  of $M \times {\mb R}$. 

Choosing coordinates on $M$ and linearizing (\ref{s21})--(\ref{s22}) at an equilibrium $(x,0)$, we obtain
\begin{equation}
\label{mat}
\begin{pmatrix}
\dot v  \\
\dot \rho 
\end{pmatrix}
=
\begin{pmatrix}
 0     &    -\nabla \mu(x) \\
 -d\mu(x)     &  0
\end{pmatrix}
\begin{pmatrix}
v  \\
\rho 
\end{pmatrix}.
\end{equation}
Since 0 is a regular value of $\mu$ by (A3), $d\mu(x)$ and $\nabla \mu(x)$ are nonzero row and column vectors respectively.  Therefore 
the matrix in \eqref{mat} has rank 2, so it has at least $n-2$ zero eigenvalues.  (Of course this is a consequence of the fact that we linearized  (\ref{s21})--(\ref{s22}) at a point on a manifold of equilibria of dimension $n-2$.)   The other eigenvalues are $\pm \|\nabla \mu(x) \|$, where $\| v \| = \| v \|_{g(x)}$; this is shown by the following calculations, which use the fact
$$
d\mu(x) v = \left< \nabla \mu(x), v \right>_{g(x)}.
$$
\begin{align*}
\begin{pmatrix}
 0     &    -\nabla \mu(x) \\
 -d\mu(x)     &  0
\end{pmatrix}
\begin{pmatrix}
\nabla \mu(x)  \\
-\|\nabla \mu(x) \|
\end{pmatrix}
&=
\begin{pmatrix}
\|\nabla \mu(x) \| \nabla \mu(x)  \\
-\| \nabla \mu(x) \|^2
\end{pmatrix}, \\
\begin{pmatrix}
 0     &    -\nabla \mu(x) \\
 -d\mu(x)     &  0
\end{pmatrix}
\begin{pmatrix}
\nabla \mu(x)  \\
\| \nabla \mu(x) \|
\end{pmatrix}
&=
\begin{pmatrix}
-\| \nabla \mu(x) \| \nabla \mu(x)  \\
-\| \nabla \mu(x) \|^2
\end{pmatrix}. 
\end{align*}
Therefore $N\times\{0\}$ is a compact normally hyperbolic manifold of equilibria for  (\ref{s11})--(\ref{s12}) with $\epsilon=0$ \cite{Jones_GSPT}.
It follows that for small $\epsilon>0$, (\ref{s11})--(\ref{s12}) has a normally hyperbolic invariant manifold $N_\epsilon$ near $N_0$.  

Locally we may assume the coordinates on $M$ are chosen so that $\mu=x_n$.  Let $y=(x_1,\ldots,x_{n-1})$, so $x=(y,x_n)$.   Then locally $N_\epsilon$ is parameterized by $y$ and is given by
$$
x_n=x_n(y,\epsilon)=x_n(x_1,\ldots,x_{n-1},\epsilon), \quad \rho=\rho(y,\epsilon)=\rho(x_1,\ldots,x_{n-1},\epsilon),
$$
with $x_n(y,0)=\rho(y,0)=0$.  After division by $\epsilon$,  the system \eqref{s11}--\eqref{s12} restricted to $N_\epsilon$ is given by  
$$
\dot y = -\nabla_yf(y,0)+{\mc O}(\epsilon) ,
$$
where $\nabla_yf(y,x_n)$ denotes the first $n-1$ components of $\nabla f(y,x_n)$.

It follows that the system \eqref{s11}--\eqref{s12} restricted to $N_\epsilon$ is a perturbation of the negative gradient flow of $(f|_{\mu^{-1}(0)}, g|_{\mu^{-1}(0)})$.  Since $(f|_{\mu^{-1}(0)}, g|_{\mu^{-1}(0)})$ is Morse-Smale by (A5), its negative gradient flow is structurally stable.  Therefore, for small $\epsilon>0$, the flow of \eqref{s11}--\eqref{s12} restricted to $N_\epsilon$ is topologically equivalent to the negative gradient flow of $(f|_{\mu^{-1}(0)}, g|_{\mu^{-1}(0)})$. 

An equilibrium $x$ of  the negative gradient flow of $(f|_{\mu^{-1}(0)}, g|_{\mu^{-1}(0)})$ corresponds, for each $\lambda>0$, to an equilibrium $(x,\eta)$ of (\ref{eqn15a})--(\ref{eqn15b}), which turn corresponds to the equilibrium $(x,\epsilon\eta)$ of (\ref{s11})--(\ref{s12}); the latter lies in $N_\epsilon$, which contains all complete orbits nearby.  It has one higher index than that of $x$ for the negative gradient flow of $(f|_{\mu^{-1}(0)}, g|_{\mu^{-1}(0)})$.  (The reason is that one of the transverse eigenvalues computed above is positive.)  If the only connections between these equilibria are those in $N_\epsilon$, the resulting Morse-Smale-Witten chain complex  is the same as that of $(f|_{\mu^{-1}(0)}, g|_{\mu^{-1}(0)})$ with degree shifted by one.

To rule out the existence of other connections, note that for $\epsilon>0$, the energy $E_\epsilon(x,\rho)=\epsilon f(x)+\rho\mu(x)$ decreases along solutions of (\ref{s11})--(\ref{s12}).  For small $\epsilon>0$, the energy difference between two equilibria $(x_-,\epsilon\eta_-)$ and $(x_+,\epsilon\eta_+)$ of (\ref{s11})--(\ref{s12}) is of order $\epsilon$.  

On the other hand, by the normal hyperbolicity of $N_0$, any sufficiently small neighborhood $V$ of $N_0$ has the property that for (\ref{s11})--(\ref{s12}) with $\epsilon$ small, a solution of (\ref{s11})--(\ref{s12}) that starts in  $V \setminus N_\epsilon$ must leave $V$ in forward or backward time.  Therefore a solution of (\ref{s11})--(\ref{s12}) that connects two equilibria but does not lie in $N_\epsilon$ must at some time leave $V$ through its boundary.  If we can show that it must do so at a point where $E_\epsilon$ is of order one, we have a contradiction.

For a small $\alpha>0$, we can take $V=\{(x,\rho) : |\mu(x)|<\alpha, |\rho|<\alpha, |\rho\mu(x)|)<\frac{\alpha^2}{4} \}$; see Figure \ref{fig:alpha_rho}.  At points on the boundary where $|\rho\mu(x)|=\frac{\alpha^2}{4}$, $|E_\epsilon|$ is close to $\frac{\alpha^2}{4}$ for small $\epsilon$.  Thus if we can show that a solution that connects two equilibria must leave $V$ through such a point, we are done.

\begin{figure}[htbp]
\centering
\includegraphics[scale=0.70]{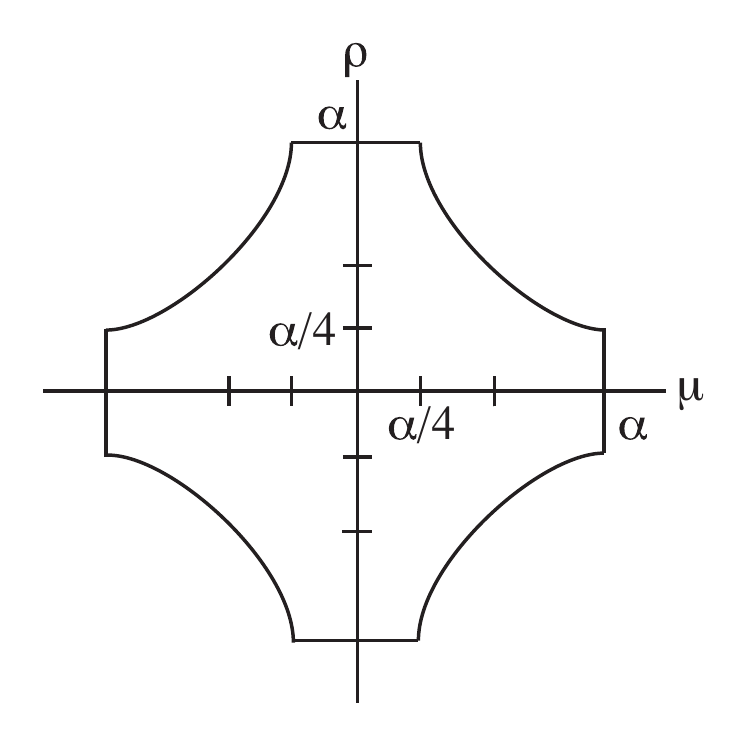}
\caption{The set $V$. }
\label{fig:alpha_rho}
\end{figure}

This is not true, but by making a small $\epsilon$-dependent alteration in $V$, we can make it true: we replace the portions of $\partial V$ on which $\mu=\pm\alpha$ or $\rho=\pm\alpha$ by nearby invariant surfaces, so that solutions cannot cross them.  More precisely, we replace the portion of the boundary on which $\mu=\pm\alpha$ by a union of integral curves of (\ref{s11})--(\ref{s12}) that start on the codimension-two surfaces $\mu=\pm\alpha$, $\rho=0$, and we replace the portion of the boundary on which $\rho=\pm\alpha$ by a union of integral curves of (\ref{s11})--(\ref{s12}) that start on the codimension-two surfaces $\mu=0$, $\rho=\pm\alpha$.  Details are left to the reader.

\section{Fast-slow system associated with the $\lambda \to 0$ limit\label{section4}}

Now we consider the limit of (\ref{eqn15a})--(\ref{eqn15b})  as $\lambda \to 0$. In this limit,  (\ref{eqn15a})--(\ref{eqn15b}) is a fast-slow system  \cite{Jones_GSPT}. There is one slow variable, $\eta$.  In this section, we will identify the slow manifold, and study the fast flow and slow equation, under appropriate generic assumptions.  We will also relate various Morse indices. 

\subsection{The slow manifold}

The {\em slow manifold} for (\ref{eqn15a})--(\ref{eqn15b}) is the set of equilibria for $\lambda=0$:
$$
{\mc C}_{\mc F} = \left\{ (x, \eta)\in M\times {\mb R} \;|\;  \nabla f(x) + \eta \nabla \mu(x) =0  \right\}.
$$
We assume:
\begin{enumerate}
\item[(A6)] ${\rm Crit}(f) \cap {\rm Crit }(\mu)$ is empty. 
\item[(A7)] $\nabla f +\eta\nabla \mu$ is a transverse section of $\pi^* TM$, where $\pi: M \times \mb R \to M$ is the projection. 
\end{enumerate}
(A7) implies that  ${\mc C}_{\mc F}$ is a 1-dimensional smooth submanifold of $M\times {\mb R}$. Indeed, it is noncompact and its end at infinity is asymptotic to $\left( {\rm Crit}\mu \times [R, +\infty) \right) \cup \left( {\rm Crit}\mu  \times (-\infty, -R]\right)$ for $R$ large. This can be seen as follows. If there is a sequence $(x_i, \eta_i)\in {\mc C}_{\mc F}$ with $\lim_{i \to \infty} |\eta_i| = \infty$, then we see that $\nabla \mu( x_i) \to 0$ and a subsequence of $x_i$ converging to some $y \in {\rm Crit}\mu$. Conversely, for any $y \in {\rm Crit}\mu$, let $S_\epsilon(y)\subset M$ denote the sphere of radius $\epsilon$ around $y$. Since by (A2) $y$ is a nondegenerate critical point of $\mu$, we see that for small $\epsilon>0$, the map
\begin{align}
S_\epsilon(y)\to S(T_y M),\ x\mapsto {\nabla\mu(x)\over \| \nabla\mu(x)\| }
\end{align}
has degree $\pm 1$. On the other hand, by (A6) we can use local coordinates near $y$ in which the vector field $\nabla f$  is constant. Hence there exist a unique pair of points $(x_\epsilon^\pm, \eta_\epsilon^\pm)$, with $x_\epsilon^\pm\in S_\epsilon(p)$, such that 
\begin{align}
(x_\epsilon^\pm, \eta_\epsilon^\pm)\in {\mc C}_{\mc F}.
\end{align}
We can order the two points so that $\lim_{\epsilon \to 0} \eta_\epsilon^\pm = \pm\infty$.

For fixed $\eta$, denote the function $f_{\eta}= f + \eta \mu: M \to {\mb R}$.  Then consider the function
\begin{align}
\begin{array}{cccc}
d_{\mc C}: & {\mc C}_{\mc F} & \to & {\mb R}\\
      & p=(x, \eta) & \mapsto & \det {\rm Hess} f_\eta (x).
\end{array}
\end{align}
We assume
\begin{enumerate}
\item[(A8)] 0 is a regular value of $d_{\mc C}$.
\item[(A9)] For any $p =(x,\eta) \in d_{\mc C}^{-1}(0)$, $\mu(x)\neq0$. Therefore $d_{\mc C}^{-1}(0) \cap {\rm Crit} ({\mc F}) = \emptyset$.
\end{enumerate}
We denote ${\mc C}_{\mc F}^{sing}= d_{\mc C}^{-1}(0)$. From (A8), as $(x, \eta)\in{\mc C}_{\mc F}$ crosses a point in ${\mc C}_{\mc F}^{sing}$, one and only one of the eigenvalues of ${\rm Hess} f_\eta(x)$ changes its sign. 

By the implicit function theorem, ${\mc C}_{\mc F}$ can be smoothly parameterized by $\eta$ near any point of ${\mc C}_{\mc F} \setminus {\mc C}_{\mc F}^{sing}$ as $x=x(\eta)$. On the other hand, near  ${\mc C}_{\mc F}^{sing}$ we have the following result.

\begin{prop}
\label{prop:singpar}
Let  $p=(x_p,\eta_p) \in {\mc C}_{\mc F}^{sing}$.  We can choose local coordinates $(x_1, \ldots, x_n)$ on $M$ such that $x_p$ corresponds to $(0,\ldots,0)$, and near $p$, ${\mc C}_{\mc F}$ is parameterized by $x_n$.  Moreover, for some $c\neq0$,
\begin{equation}
\label{eq:singpar}
\eta= \eta_p + cx_n^2 + O(x_n^3).
\end{equation}
Thus $\eta |_{{\mc C}_{\mc F}}$ has a nondegenerate critical point at $p$.
\end{prop}

This is a standard result, but we shall give a proof in Subsection \ref{sec:slow}.

\subsection{The fast flow\label{sec:ff}}

The fast flow $\Phi_t: M \times {\mb R}\to M \times {\mb R}$ of (\ref{eqn15a})--(\ref{eqn15b}) is the flow on $M\times {\mb R}$ determined by (\ref{eqn15a})--(\ref{eqn15b}) for $\lambda = 0$. The set of equillibria of the fast flow is just ${\mc C}_{\mc F}$, and along the fast flow, $\eta$ is constant.

If $p=(x_p,\eta_p) \in {\mc C}_{\mc F}$, then $x_p$ is a critical point of the function $f_{\eta_p} : M \to {\mb R}$. A {\it fast solution} from $p_-\in {\mc C}_{\mc F}$ to $p_+\in {\mc C}_{\mc F}$ is a solution $\wt{x}(t)$ of (\ref{eqn15a})--(\ref{eqn15b}) for $\lambda = 0$ such that $\lim_{t\to \pm\infty} \wt{x}(t) = p_\pm$. A {\it fast orbit} from $p_-$ to $p_+$ is an equivalence class of fast solutions from $p_-$ to $p_+$ modulo time translation. A fast solution or orbit is {\em trivial} if the orbit consists of a single point.

For $p= (x_p, \eta_p) \in {\mc C}_{\mc F}$, $x_p$ is a nondegenerate critical point if and only if $p\in  {\mc C}_{\mc F} \setminus  {\mc C}_{\mc F}^{sing}$.  Note that the Morse index of a nondegenerate critical point $x$ of $f_{\eta}$ can be defined as the number of negative eigenvalues of the matrix ${\rm Hess} f_\eta(x)$.  We define the Morse index of a degenerate critical point the same way.
We shall use the already-introduced notation ${\rm index}\, (x_p, f_{\eta_p})$ to denote the  Morse index in this sense of $x_p$ as a critical point of $f_{\eta_p}$.

${\rm Index}\, (x_p, f_{\eta_p})$ is locally constant on $ {\mc C}_{\mc F}\setminus{\mc C}_{\mc F}^{sing}$. Near $p \in {\mc C}_{\mc F}^{sing}$, the values of ${\rm index}\, (x_p, f_{\eta_p})$ on the two branches differ by one, and ${\rm index}\, (x_p, f_{\eta_p})$ is the lower of these two numbers.

\subsubsection{Transversality assumptions}

On ${\mc C}_{\mc F}\setminus {\mc C}_{\mc F}^{sing}$, the fast flow is normally hyperbolic. For any subset $\beta$ of ${\mc C}_{\mc F} \setminus {\mc C}_{\mc F}^{sing}$ define
\begin{align}
 W^u(\beta)=\left\{ (x, \eta)\in M \times {\mb R}\ |\  \lim_{t\to -\infty} \Phi_t(x, \eta) \in \beta \right\}
\end{align}
and
\begin{align}
  W^s(\beta)=\left\{ (x, \eta)\in M \times {\mb R}\ |\  \lim_{t\to +\infty} \Phi_t(x, \eta) \in  \beta \right\}.
\end{align}
Note that in this section we always have $\lambda=0$, so we will not use the notation of Subsection \ref{sec:msw} to specify a value of $\lambda$. If $\beta$ is connected, define ${\rm index}\, (\beta)$ to be ${\rm index}\, (x_p, f_{\eta_p})$ for any $p=(x_p, \eta_p)$ in $\beta$. 

We assume 
\begin{enumerate}
\item[(A10)]
$W^u\left( {\mc C}_{\mc F}\setminus {\mc C}_{\mc F}^{sing} \right)$ and $W^s\left( {\mc C}_{\mc F}\setminus {\mc C}_{\mc F}^{sing}\right)$ intersect transversally in $M\times {\mb R}$. 

\item[(A11)] For each $p = (x_p, \eta_p) \in {\rm Crit}({\mc F})$, the pair $(f_{\eta_p}, g)$ is a Morse-Smale pair on $M$.
\end{enumerate}
Assumption (A10) can be thought of as a weak version of the Morse-Bott-Smale transversality condition (see \cite{Morse_Bott_Homology}) for a Morse-Bott function and a metric. (A10) does not imply (A11).  However, (A11) implies that if $p = (x_p, \eta_p) \in {\rm Crit}({\mc F})$, then for any pair of critical points $y_-, y_+\in M$ of the function $f_{\eta_p}$, the unstable manifold of $y_-$ and the stable manifold of $y_+$ intersect transversally in $M$.  This implies that the transversality in (A10) holds for $\eta$ near $\eta_p$.

There are two special types of nontrivial fast orbits that play important roles in our construction. We introduce them in the remainder of this subsection.

\subsubsection{Handle-slides and cusp orbits}

By (A10) we see, if $\beta_1, \beta_2\subset {\mc C}_{\mc F}\setminus {\mc C}_{\mc F}^{sing}$ are connected, then
\begin{align}
{\rm dim}\, \left( W^u(\beta_1) \cap W^s(\beta_2)  \right) = {\rm index}\, (\beta_1)-{\rm index}\, (\beta_2) + 1.
\end{align}
In particular, when $ {\rm index}\, (\beta_1)={\rm index}\, (\beta_2)$, the dimension is one, so the intersection consists of discrete nontrivial fast orbits. We call such an orbit a {\em handle-slide}. If a handle-slide is contained in $M \times \{ \eta\}$, we say that a handle-slide happens at $\eta$.

We assume:
\begin{enumerate}
\item[(A12)] The composition $\pi: {\rm Crit} ({\mc F})\cup {\mc C}_{\mc F}^{sing} \hookrightarrow M \times {\mb R} \to {\mb R}$ is injective. Moreover, handle-slides happen at distinct values of $\eta$ that are not in the image of $\pi$. 
\end{enumerate}

Now we look near a point $p = (x_p, \eta_p) \in {\mc C}_{\mc F}^{sing}$. All the eigenvalues of ${\rm Hess} f_{\eta_p}(x_p)$ have absolute values great than a constant $a>0$, except for the one that changes its sign at $p$. Then we define $W^u_a(p)$ to be the space of maps $u: (-\infty, 0] \to M$ satisfying
\begin{enumerate}
\item $(u(t),\eta_p) = \Phi_t(u(0),\eta_p)$ for $t\leq 0$;
\item For sufficiently large $T$, for $t\in (-\infty, -T]$, 
$$
u(t) = \exp_{x_p} V(t), \quad V\in W^{1, 2}_a((-\infty, -T], T_{x_p} M ). 
$$
Here for any open subset $\Omega \subset {\mb R}$, the space $W^{1, 2}_a(\Omega)$ is the space of functions $f$ with $ e^{a|t|} f \in W^{1, 2}(\Omega)$. 
\end{enumerate}
Similarly we can define $W^s_a(p)$.  By the map $u\mapsto u(0)$, $W^u_a(p)$ and $W^s_a(p)$ are naturally identified with smooth submanifolds of $M$, which are called the $a$-exponential unstable and stable manifolds of $p$, having dimensions
\begin{align}
{\rm dim}\, W^u_a(p) = {\rm index}\, (x_p, f_{\eta_p}),\  {\rm dim}\, W^s_a(p) =  n - {\rm index}\, (x_p, f_{\eta_p}) - 1.
\end{align}

Let $\mathring{W}^u(p) =  W^u(p) \setminus W^u_a(p)$ and $\mathring{W}^s(p) = W^s(p) \setminus W^s_a(p)$.  These manifolds are unions of orbits that approach  $p$ polynomially rather than exponentially. They are  submanifolds of $M \times \{\eta_p\}$ of dimension $ {\rm dim}\, W^u_a(p)+1$ and  ${\rm dim}\,W^s_a(p)+1$. We assume
\begin{enumerate}
\item[(A13)]For each $p = (x_p, \eta_p) \in {\mc C}_{\mc F}^{sing}$ and $q= (x_q, \eta_q) \in {\mc C}_{\mc F}$ with $\eta_p = \eta_q = \eta$, $W^u_a(x_p)$ and $W^s(x_q)$  (respectively $W^s_a(x_p)$ and $W^u(x_q)$) intersect transversely in $M \times \{\eta\}$, $\mathring{W}^u(x_p)$ and $W^s(x_q)$  (respectively $\mathring{W}^s(x_p)$ and $W^u(x_q)$) intersect transversally in $M \times \{\eta\}$.
\end{enumerate}
The reader may refer to \cite{Floer_unregularized} and \cite{Floer_intersection} for this transversality in the case of Lagrangian intersections. 

(A13) implies that for $p = (x_p, \eta_p) \in {\mc C}_{\mc F}^{sing}$ and  $q = (x_q, \eta_q) \in {\mc C}_{\mc F}\setminus {\mc C}_{\mc F}^{sing}$, $W^u_a(p) \cap W^s(q) = \emptyset$ if ${\rm index}\, (x_q, f_{\eta_q}) \ge {\rm index}\, (x_p,f_{\eta_p})-1$;  and $W^u(q) \cap W^s_a(p) = \emptyset$ if ${\rm index}\, (x_p, f_{\eta_p}) \ge {\rm index}\, (x_q,f_{\eta_q})-2$.  For example, $W^u(q) \cap W^s_a(p) = \emptyset$ if 
\begin{multline}
0 \ge {\rm dim}\ W^u(q) + {\rm dim}\ W^s_a(p)-(n+1)= {\rm index}\, (x_q,f_{\eta_q}) + (n-{\rm index}\, (x_p, f_{\eta_p})-1) \\- (n+1)= {\rm index}\, (x_q,f_{\eta_q}) -{\rm index}\, (x_p, f_{\eta_p})-2.
\end{multline}

When the inequalities are equalities, $W^u(p) \cap W^s(q) = \mathring{W}^u(p) \cap W^s(q)$ and $W^u(q) \cap W^s(p)= W^u(q) \cap \mathring{W}^s(p)$) consist of isolated orbits corresponding to solutions that approach $p$ like $\frac{1}{t}$ rather than exponentially. We call such orbits {\em cusp orbits}; parametrized ones are {\it cusp solutions}. An argument similar to one in \cite{Floer_intersection} shows that between $p$ and $q$ there are only finitely many cusp orbits.

%\st{I suggest omitting the following remark, this was discussed at the start of sec. 2.3.}

%\begin{rem}
%We have stated all assumptions (A1)-(A13) on the triple $(f, \mu, g)$. As we have mentioned, if we fix $\mu$ and $g$, then for generic smooth function $f$, all these assumptions are satisfied.
%\end{rem}

\subsubsection{Short fast orbits}

By Proposition \ref{prop:singpar}, near $p \in {\mc C}_{\mc F}^{sing}$, one can parametrize ${\mc C}_{\mc F}$ by $\gamma_p: (-\epsilon, \epsilon) \to {\mc C}_{\mc F}$ such that
\begin{align}\label{eqn510}
\gamma_p(0) = p,\ \eta(t) = \eta_p \pm s^2. 
\end{align}
We choose the orientation of $\gamma_p$ such that ${\rm index}\left( \gamma_p((-\epsilon, 0)) \right) = {\rm index} \left(\gamma_p ((0, \epsilon)) \right) + 1$. For $\epsilon$ small enough there exists a unique such parametrization $\gamma$ and we call $\gamma$ the {\em canonical parametrization} near $p$. Then, for $\epsilon$ small enough, there is a unique orbit of the flow of $-\nabla f_{\eta(s)}$ from $\gamma(-s)$ to $\gamma(s)$ and they are all contained in a small neighborhood of $x_p$. 

\begin{lem}
There exists $\epsilon_0>0$ such that for all $p \in {\mc C}_{\mc F}^{sing}$ and for $s \in (0, \epsilon_0)$ there exists a unique orbit of the fast flow from $\gamma_p(-s)$ to $\gamma_p(s)$, where $\gamma_p: (-\epsilon_0, \epsilon_0)\to {\mc C}_{\mc F}$ is the canonical parametrization of ${\mc C}_{\mc F}$ near $p$.
\end{lem}
If $0< s \leq \epsilon< \epsilon_0$, then we call a fast orbit from $\gamma_p(-s)$ to $\gamma_p(s)$ an {\em $\epsilon$-short orbit}.
\begin{lem}\label{lemma_smallenergy} For any $\epsilon$ small enough, there exists $\delta= \delta(\epsilon)>0$ such that, for all $\eta\in {\mb R}$, any orbit of the flow of $-\nabla f_\eta$ which is not an $\epsilon$-short orbit has energy no less than $\delta$.
\end{lem}

\begin{proof} Denote ${\mc C}_{\mc F}^{sing, \epsilon} = \cup_{p \in {\mc C}_{\mc F}^{sing}} \gamma_p((-\epsilon, \epsilon))$. Because we have assumed that the map ${\mc C}_{\mc F}^{sing}\to M \times {\mb R}\to {\mb R}$ is injective, any fast orbit both of whose beginning and end lie in ${\mc C}_{\mc F}^{sing, \epsilon}$ much be an $\epsilon$-short orbit. So suppose the lemma doesn't hold, then there exists $\epsilon>0$ and a sequence of nontrivial fast orbits $(y_k, \eta_k) \subset M \times \{\eta_k\}$ such that 
\begin{align}\label{317}
\lim_{k\to \infty} f_{\eta_k}( y_k(-\infty)) - f_{\eta_k}( y_k(+\infty)) = 0
\end{align}
and they are not $\epsilon$-short ones. Hence without loss of generality, we may assume that for all $k$, $y_k(-\infty) \notin {\mc C}_{\mc F}^{sing, \epsilon}$. 

\begin{enumerate}
\item $\lim_{k\to \infty}\eta_k = \pm \infty$, which implies that $y_k(\pm\infty)$ converges to points in ${\rm Crit} (\mu)$. However, if we rescale the orbit by $z_k(t) = y_k( \eta_k^{-1} t)$, then $z_k$ converges to a (possibly broken) nontrivial orbit of the flow of $-\nabla \mu$, which contradicts with (\ref{317}).

\item $\lim_{k \to \infty} \eta_k = \eta_\infty\in {\mb R}$. Then a subsequence of $y_k$ converges to a broken orbit of $f_{\eta_\infty}$, which must be a constant one. Hence
\begin{align}
\lim_{k\to \infty}(y_k(\pm\infty)) := y_\infty\in {\mc C}_{\mc F}\setminus {\mc C}_{\mc F}^{sing, \epsilon}.
\end{align}
So for $k$ large enough, $y_k(\pm\infty)$ lie in the same connected component of ${\mc C}_{\mc F} \setminus {\mc C}_{\mc F}^{sing, {\epsilon\over 2}}$. This is impossible because on each such component, distinct points have distinct values of $\eta$.
\end{enumerate}
\end{proof}

\subsection{The slow equation\label{sec:slow}}

On ${\mc C}_{\mc F}\setminus {\mc C}_{\mc F}^{sing}$, which is parameterized by $\eta$, the slow equation is given by restricting (\ref{eqn15b}) to ${\mc C}_{\mc F}\setminus {\mc C}_{\mc F}^{sing}$ and dividing by $\lambda^2$:
\begin{equation}
\label{slowde}
\eta^\prime=-\mu(x(\eta)).
\end{equation}
Thus $\eta^\prime$ changes sign only when $\mu=0$.
We have
$${\rm Crit} ({\mc F}) =  \left\{ (x, \eta)\in {\mc C}_{\mc F} \;|\;  \mu(x)=0 \right\}.$$
Thus $\eta^\prime$ changes sign only at points of ${\rm Crit} ({\mc F})$.  By (A9), $\eta^\prime$ does not change sign at points of  ${\mc C}_{\mc F}^{sing}$.

\begin{prop}
\label{prop:sloweq}
Equilibria of (\ref{slowde}) are hyperbolic.
\end{prop}

\begin{proof}
  Let  $p=(x_p,\eta_p)$ be an equilibrium of (\ref{slowde}). Then $\mu(x_p)=0$, so by (A3), $d\mu(x_p)\neq0$. Hence we can choose local coordinates $(x_1,\ldots,x_n)$ near $x_p$ such that $\mu(x_1,\ldots,x_n)=x_n$.   Near $p$,  ${\mc C}_{\mc F}$ can be defined without reference to the metric $g$ by the equations
\begin{align}\label{42}
{\partial f\over \partial x_i} + \eta \delta_{nj} = 0,\ j=1, \ldots, n; \quad \delta_{ij}=\mbox{Kronecker delta}.
\end{align}
Let $e_n=\begin{pmatrix} 0 & \ldots &0&1\end{pmatrix}^T$, the $n$th standard unit basis vector in ${\mb R}^n$.  Then the derivative of the system (\ref{42}) is the $n \times (n+1)$ matrix
\begin{align}\label{12}
\begin{pmatrix}
{\rm Hess} f(x) & e_n
\end{pmatrix}
=
\begin{pmatrix}
 A & b & 0\\
 b^T & c & 1\\
\end{pmatrix} , 
\end{align}
where $A$ is an $(n-1)\times(n-1)$ matrix, and $( b^T, c)  = \left( {\partial^2 f\over \partial x_1 \partial x_n}, \cdots,  {\partial^2 f \over \partial x_{n-1} \partial x_n}, {\partial^2 f\over \partial x_n^2}  \right)$.  

Since $x_p$ is a critical point of $f |_{ \mu^{-1}(0)}$, by (A5), $A$ is invertible. The matrix (\ref{12}) has full rank; a nonzero vector in its kernel is tangent at $p$ to ${\mc C}_{\mc F}$.  Since $p\in{\mc C}_{\mc F}\setminus {\mc C}_{\mc F}^{sing}$, such a vector has nonzero $\eta$-component, and one can check that it has nonzero $x_n$-component.  Therefore near $p$, when ${\mc C}_{\mc F}$ is parameterized by $\eta$, we have $x(\eta)=(y(\eta),x_n(\eta))$ with 
$$
x_n(\eta)=a(\eta-\eta_p)+\mathcal{O}(\eta-\eta_p)^2, \quad a\neq0.  
$$
But then (\ref{slowde}) becomes
$$
\eta^\prime = a(\eta-\eta_p)+\mathcal{O}(\eta-\eta_p)^2,
$$
which proves the result.
\end{proof}

Next we discuss the slow equation near points of ${\mc C}_{\mc F}^{sing}$.  Recall that if $p_0$ is an equilibrium of $\dot p=h(p)$, $h$ is $C^s$,
$Dh(p_0)$ has $m$ eigenvalues (counting multiplicity) with real part 0, and $E$ is the corresponding $m$-dimensional invariant subspace of $\dot v=Dh(p_0)v$, then there is a $C^s$ invariant manifold through $p_0$ and tangent there to $E$, called the {\em center manifold} of $\dot p=h(p)$ at $p_0$.
\begin{prop}
\label{prop:singpar2}
Let  $p=(x_p,\eta_p) \in {\mc C}_{\mc F}^{sing}$.  We can choose local coordinates $(x_1, \ldots, x_n)$ on $M$ near $x_p$, with $x_p$ corresponding to 0, in which 
\begin{equation}
\label{dgradf}
D\nabla f(0)=
\begin{pmatrix}
A      &  0  \\
 0     &  0
\end{pmatrix}, \quad \mbox{$A$ an invertible symmetric  $(n-1)\times(n-1)$ matrix.}
\end{equation}
The center manifold of (\ref{eqn15a})--(\ref{eqn15b}) at $p$ is parameterized by $(x_n,\eta)$.  There are numbers $c\neq0$ and $d\neq0$ such that the system (\ref{eqn15a})--(\ref{eqn15b}), restricted to the center manifold, is 
\begin{align}
 x_n^\prime &= c(\eta-\eta_p)+d x_n^2+\ldots, \label{cmdex} \\
\eta^\prime &= -\lambda^2(\mu(x_p)+\ldots), \label{cmdeeta}
\end{align}
where  \ldots indicates terms of higher order (and, in (\ref{cmdex}), other second-order terms).
Therefore near $p$, ${\mc C}_{\mc F}$ is parameterize by $x_n$, and $\eta | {\mc C}_{\mc F}$ has a nondegenerate critical point at $p$.
\end{prop}

\begin{proof}
(A7) implies that we can choose local coordinates $(x_1, \ldots, x_n)$ on $M$ near $x_p$, with $x_p$ corresponding to 0,  in which (\ref{dgradf}) holds.  In these coordinates, $\mu=\mu(x_p)+a^Tx + \dots$ with $a\in{\mb R}^n$.  
Write $x=(y,x_n)$ with $y\in{\mb R}^{n-1}$, and write $a=(b,c)$ with $b\in{\mb R}^{n-1}$ and $c\in{\mb R}$.  Near $p$ the system (\ref{eqn15a})--(\ref{eqn15b}) becomes
\begin{align}
y^\prime &= Ay + b(\eta-\eta_p) + \ldots, \label{lc1}\\ 
x_n^\prime &= c(\eta-\eta_p)+dx_n^2 + \ldots,  \label{lc2} \\ 
\eta^\prime &= -\lambda^2(\mu(x_p)+\ldots),  \label{lc3}
\end{align}
where \dots indicates higher-order terms (and, in (\ref{lc2}), other second-order terms).  (A7) implies that $c\neq0$, and (A8) implies that $d\neq0$.
The linearization of (\ref{lc1})--(\ref{lc3}) at $(y,x_n,\eta)=(0,0,\eta_p)$ is 
$$
\begin{pmatrix}
y^\prime  \\ 
x_n^\prime \\ 
\eta^\prime 
\end{pmatrix}
\begin{pmatrix}
A      &   0 & b \\
  0    &  0 & c \\
  0    &  0 &0
\end{pmatrix}
\begin{pmatrix}
y  \\ 
x_n \\ 
\eta-\eta_p \end{pmatrix}.
$$ 
The matrix has a two-dimensional generalized eigenspace for the eigenvalue 0 that is tangent to the center manifold of (\ref{lc1})--(\ref{lc1}) at $p$.  The center manifold  is parameterized by $(x_n,\eta)$, and to lowest order it is given by
$$
y=-A^{-1}b(\eta-\eta_p)
$$
The restriction of  (\ref{lc1})--(\ref{lc3}) to the center manifold is therefore given by (\ref{cmdex})--(\ref{cmdeeta}).
\end{proof}

Proposition \ref{prop:singpar} is of course just a partial restatement of Proposition \ref{prop:singpar2}.

\subsection{Three indices}

In this subsection we describe the relation between three different Morse indices:
\begin{enumerate} 
\item For any $p = (x_p, \eta_p) \in {\mc C}_{\mc F}$, we have ${\rm index} (x_p, f_{\eta_p})$, used in Subsection \ref{sec:ff}.
\item For $p = (x_p, \eta_p) \in {\rm Crit} ({\mc F})$, we can consider ${\rm index} \,(p, {\mc F})$. 
\item If $p=(x_p, \eta_p) \in {\rm Crit} ({\mc F})$,  then $x_p \in {\rm Crit} (f|_{\mu^{-1}(0)})$, and we can consider ${\rm index}\, (x_p, f|_{\mu^{-1}(0)})$. 
\end{enumerate}
We have already seen in Lemma \ref{lemma21} that ${\rm index}(p, {\mc F}) = {\rm index}(x_p, f|_{\mu^{-1}(0)}) + 1$ for $p = (x_p, \eta_p) \in {\rm Crit}({\mc F})$. The best way to see the relation of the different indices is to look at the corresponding unstable manifolds. 
\begin{prop}
\label{twoindexprop2}
Let $p= (x_p, \eta_p) \in {\rm Crit}({\mc F})$.
\begin{enumerate}
\item If $p$ is a repeller of the slow equation, then $ {\rm index}\, (p, {\mc F})={\rm index} (x_p, f_{\eta_p}) +1$.
\item If $p$ is an attractor of the slow equation, then ${\rm index}\, (p, {\mc F})={\rm index} (x_p, f_{\eta_p})$.
\end{enumerate}
\end{prop}

\begin{proof}
If $p$ is a repeller of the slow equation, then for small $\lambda>0$, the equilibrium $p$ of (\ref{eqn15a})--\ref{eqn15b}) has one more positive eigenvalue than the equilibrium $x_p$ of  $f_{\eta_p}$.  Since the index of a critical point of a Morse function $h$ is the number of positive eigenvalues of the corresponding equilibrium of $\nabla h$ for any metric, (1) follows.  (2) is similar.
\end{proof}

\section{Adiabatic limit $\lambda\to 0$}\label{section5}

\subsection{Slow solutions, fast solutions, and fast-slow solutions}

Let $p_-, p_+\in {\mc C}_{\mc F}$. 

Recall that a nontrivial fast solution from $p_-$ to $p_+$ is a nonconstant solution $\wt{x}(t)$, $-\infty<t<\infty$, of (\ref{eqn15a})--(\ref{eqn15b}) for $\lambda = 0$, such that $\lim_{t\to \pm\infty} \wt{x}(t) = p_\pm$. 

Let $(x(\eta),\eta)$ parameterize the closure of a component of $ {\mc C}_{\mc F}\setminus {\mc C}_{\mc F}^{sing}$. Let $I(\alpha_-, \alpha_+)\subset {\mb R}$ be the closed interval from $\alpha_-$ to $\alpha_+$,  with $\alpha_- \le \alpha_+$; we allow $\alpha_-= -\infty$, $\alpha_+=+\infty$, and $\alpha_- = \alpha_+ \in {\mb R}$. Let $\eta(t)$, $t$ in the interior of $I(\alpha_-, \alpha_+)$, be a solution of $\eta^\prime=-\mu(x(\eta))$. If $\alpha_-$ (respectively $\alpha_+$) is finite, we extend $\eta(t)$ continuously to $\alpha_-$ (respectively $\alpha_+$).  Let $p(t)=(x(\eta(t)),\eta(t))$, $t\in I(\alpha_-, \alpha_+)$, and let $p_\pm \in {\mc C}_{\mc F}$.  Then $p(t)$ is a {\em slow solution from $p_-$ to $p_+$} (for short, a {\em slow solution}) provided:
\begin{enumerate}
\item if $\alpha_-$ (respectively $\alpha_+$) is finite, then $p(\alpha_-)=p_-$ (respectively $p(\alpha_+)=p_+$);
\item if $\alpha_-=-\infty$ (respectively $\alpha_+=+\infty$), then 
 $\lim_{t\to-\infty }p(t)=p_-$ (respectively  \newline $\lim_{t\to+\infty }p(t)=p_+$).
\end{enumerate}
A slow solution or its orbit is {\it trivial} if the orbit is a single point in ${\mc C}_{\mc F}$.  

A {\em fast-slow solution from $p_-$ to $p_+$} of (\ref{eqn15a})--(\ref{eqn15b}) (for short, a {\em fast-slow solution}) is a sequence 
\begin{align}\label{eqn61}
{\ms X} = \left( p_0, \sigma_1, p_1, \sigma_2, p_2, \ldots, p_{n-1}, \sigma_n, p_n\right)
\end{align}
such that: 
\begin{enumerate}
\item $p_0=p_-$, $p_n=p_+$.
\item Each $p_i \in {\mc C}_{\mc F}$.
\item Each $\sigma_i$ is a nontrivial fast solution or a slow solution from $p_{i-1}$ to $p_i$.  Trivial slow solutions are allowed, but $\sigma_1$ and $\sigma_n$ are not allowed to be trivial slow solutions.  
\item Either (i) $\sigma_i$ is slow for $i$ even and fast for $i$ odd, or (ii) $\sigma_i$ is slow for $i$ odd and fast for $i$ even.
\end{enumerate}

Trivial slow solutions are allowed to deal with the possibility that a fast solution to $p$ is followed by a fast solution from $p$.

An orbit (fast orbit, slow orbit, or fast-slow orbit) is an equivalence class of solutions obtained by forgetting the parametrization (but remembering the orientation in which ${\mc F}$ decreases). In particular, a fast-slow orbit is denoted by $\left[ {\ms X}\right]$ for ${\ms X}$ as in (\ref{eqn61}). 

\begin{rem}
A fast-slow orbit defined above is similar to a ``flow line with cascades'' in Morse-Bott theory (see \cite{Frauenfelder_thesis}), where the fast solutions correspond to  ``cascades''. The difference is that the slow manifold ${\mc C}_{\mc F}$ is not everywhere normally hyperbolic, and we can have $p_i\in {\mc C}_{\mc F}^{sing}$. 
\end{rem}

In the remainder of this subsection, we will discuss some properties of fast-slow orbits from $p_- \in {\mc C}_{\mc F}$ to $p_+ \in {\mc C}_{\mc F}$. We are particularly interested in the case $p_\pm  \in {\rm Crit}\left( {\mc F} \right)$ and ${\rm index}(p_-, {\mc F})- {\rm index}(p_+ , {\mc F}) = 1$.

The following proposition is obvious.

\begin{prop}
Consider a nontrivial slow solution $p(t)$, $t\in I(\alpha_-, \alpha_+)$, from $p_-$ to $p_+$.
\begin{enumerate}
\item If $p_- \in {\mc C}_{\mc F}^{sing}$, then  either (i) for all $t\in (\alpha_-,\alpha_+)$, ${\rm index}\, (x_{p(t)}, f_{\eta_{p(t)}})={\rm index}\, (x_-, f_{\eta_-})$, or (ii) for all $t\in (\alpha_-, \alpha_+)$, ${\rm index}\, (x_{p(t)}, f_{\eta_{p(t)}})=\newline{\rm index}\, (x_-, f_{\eta_-})+1$.
\item If $p_+ \in {\mc C}_{\mc F}^{sing}$, then either  (i) for all $t\in (\alpha_-, \alpha_+)$, ${\rm index}\, (x_{p(t)}, f_{\eta_{p(t)}})={\rm index}\, (x_+, f_{\eta_+})$, or (ii) for all $t\in (\alpha_-, \alpha_+)$, ${\rm index}\, (x_{p(t)}, f_{\eta_{p(t)}})=\newline{\rm index}\, (x_+, f_{\eta_+})+1$.
\end{enumerate}
\end{prop}

A slow solution $p(t)$ for $t\in I(\alpha_-, \alpha_+)$ is {\em regular} if it is nontrivial and satisfies: 
\begin{enumerate}
\item If $p_- \in {\mc C}_{\mc F}^{sing}$, then  for all $t\in (\alpha_-,\alpha_+)$, ${\rm index}\, (x_{p(t)}, f_{\eta_{p(t)}}) = \newline{\rm index}\, (x_-, f_{\eta_-})+1$. 

\item If $p_+ \in {\mc C}_{\mc F}^{sing}$, then  for all $t\in(\alpha_- ,\alpha_+)$, ${\rm index}\, (x_{p(t)}, f_{\eta_{p(t)}}) = {\rm index}\, (x_+, f_{\eta_+})$. 
\end{enumerate}
A slow orbit is regular if it is the orbit of a regular slow solution. 

See Figure \ref{fig:regular}.

\begin{figure}[htbp]
\centering
\includegraphics[scale=0.7]{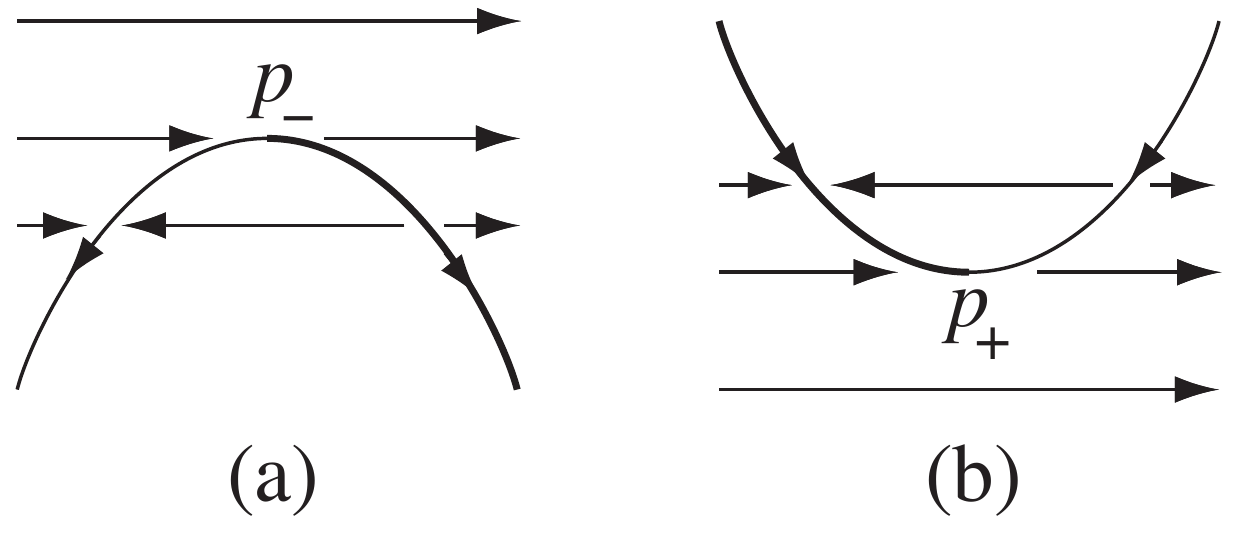}
\caption{(a) Direction of slow flow is downward.  The thick curve is a regular slow orbit that starts at $p_-$.  Points on it have index one greater than $p_-$.  (b) Direction of slow flow is downward. The thick curve is a regular slow orbit that ends at $p_+$.  Points on it have the same index as $p_+$.}
\label{fig:regular}
\end{figure}

\begin{prop}
\label{prop:indexchange}
Let $\gamma$ be either a slow or fast solution from $p_-=(x_-,\eta_-)$ to $p_+=(x_+,\eta_+)$.
\begin{enumerate}
\item If $\gamma$ is a regular slow solution, then:
\begin{enumerate}
\item If $p_- \in {\mc C}_{\mc F}^{sing}$, then ${\rm index}\, (x_+, f_{\eta_+})={\rm index}\, (x_-, f_{\eta_-})+1$.
\item Otherwise ${\rm index}\, (x_+, f_{\eta_+})={\rm index}\, (x_-, f_{\eta_-})$ 
\end{enumerate}
\item If $\gamma$ is a nontrivial fast solution (so $\eta_+=\eta_-$), then:
\begin{enumerate}
\item At most one of $p_\pm$ is in ${\rm Crit}({\mc F}) \cup  {\mc C}_{\mc F}^{sing}$.
\item ${\rm index}\, (x_+, f_{\eta_+}) \le {\rm index}\, (x_-, f_{\eta_-})$.
\item If $p_+\in{\rm Crit}({\mc F}) \cup  {\mc C}_{\mc F}^{sing}$ or $p_- \in {\rm Crit}({\mc F})$, then ${\rm index}\, (x_+, f_{\eta_+}) \le {\rm index}\, (x_-, f_{\eta_-})-1$.
\end{enumerate}
\end{enumerate}
\end{prop}

\begin{proof} 
The result for slow solutions is immediate from their definition.

For fast solutions, (2)(a) follows from assumption (A12)

Let $p(t)$ be a fast solution, let ${\rm index}\, (x_-, f_{\eta_-})=k$, and ${\rm index}\, (x_+, f_{\eta_+})=l$.  (Of course $\eta_+=\eta_-$.)

If  $p_+\in{\rm Crit}({\mc F}) \cup  {\mc C}_{\mc F}^{sing}$, then  (A11) or (A13) implies that $W^u(p_-)$ and $W^s(p_+)$ intersect transversally within $M\times\{\eta_-\}$.  We have $\dim W^u(p_-)=k$, and $\dim W^s(p_+)=n-l$, so the intersection is nonempty provided $k + (n-l) \ge n+1$, i.e., $l \le k-1$. 

If  $p_-\in{\rm Crit}({\mc F})$, the same argument applies.  However, if $p_-\in {\mc C}_{\mc F}^{sing}$, then $\dim W^u(p_-)=k+1$.   We again have $\dim W^s(p_+)=n-l$, so the intersection is nonempty provided $(k+1) + (n-l) \ge n+1$, i.e., $l \le k$.

If  neither $p_-$ nor $p_+$ is in ${\rm Crit}({\mc F}) \cup  {\mc C}_{\mc F}^{sing}$, then we do not know that $W^u(p_-)$ and $W^s(p_+)$ intersect transversally within $M\times\{\eta_-\}$.  However, by (A10), $(k+1) + (n-l+1) \ge n+2$, so $l\le k$.
\end{proof}

\begin{prop}
\label{prop:indexequal}
Consider a fast-slow solution $${\ms X}= \left( p_0, \sigma_1, p_1, \sigma_2, p_2, \ldots, p_{n-1},\sigma_n, p_n \right)$$ from $p_- = (x_-, \eta_-) $ to $p_+=(x_+, \eta_+)$, with neither in  ${\mc C}_{\mc F}^{sing}$ and ${\rm index}\, (x_+, f_{\eta_+})\ge{\rm  index}(x_-, f_{\eta_-})=k$.  Then 
\begin{enumerate}
\item For all $i\geq 1$, ${\rm index}\, (x_{p_{i}}, f_{p_{i}})=k$, except for the case where $p_i \in {\mc C}_{\mc F}^{sing}$ and $\sigma_i$ is a fast solution, in which case ${\rm index}\, (x_{p_{i}}, f_{p_{i}})=k-1$. In particular, ${\rm index}\, (x_+, f_{\eta_+})=k$.
\item All slow solutions appearing in ${\ms X}$ are regular.  In particular, they are nontrivial.
\item If $p_i\in{\rm Crit}({\mc F})$, then $i=0$ or $i=n$.  In the first case, $p_0=p_- \in {\rm Crit}^-({\mc F})$ and $\sigma_1$ is slow; in the second case, $p_n = p_+ \in {\rm Crit}^+({\mc F})$ and $\sigma_n$ is slow.
\item The fast solutions are handle-slides or cusp solutions.
\end{enumerate}
\end{prop}

\begin{proof}
Along a fast-slow solution, ${\rm index}\, (x_{p_{i}}, f_{p_{i}}) \le {\rm index}\, (x_{p_{i-1}}, f_{p_{i-1}})$, except along a regular slow solution for which $p_{i-1} \in {\mc C}_{\mc F}^{sing}$ and $p_i\in {\mc C}_{\mc F} \setminus {\mc C}_{\mc F}^{sing}$; in this case  ${\rm index}\, (x_{p_{i}}, f_{p_{i}}) = {\rm index}\, (x_{p_{i-1}}, f_{p_{i-1}})+1$.  However, the index drops along the fast solution from $p_{i-2}$ to $p_{i-1}$.  Thus (1) holds; also, nontrivial slow solutions must be regular and (3) must hold, otherwise we would have ${\rm index}\, (x_+, f_{\eta_+})<{\rm  index}(x_-, f_{\eta_-})$.  (3) rules out fast solutions that start or end in ${\rm Crit}({\mc F})$; any remaining fast solutions are handle-slides or cusp solutions, so (4) holds.  Since handle-slides and cusp solutions occur at different values of $\eta$, slow solutions must be nontrivial, which completes the proof of (2).
\end{proof}

Now we describe fast-slow solutions between two critical points whose indices differ by 1.  Recall from Subsection \ref{sec:mh} that  ${\rm Crit}_k({\mc F})$ denotes the set of index $k$ critical points of ${\mc F}$. Let ${\rm Crit}_k^+({\mc F})$ (respectively ${\rm Crit}_k^-({\mc F})$) denote the set of equilibria in  ${\rm Crit}_k({\mc F})$ that are stable (respectively unstable) equilibria of the slow equation. Let ${\rm Crit}^\pm({\mc F}) = \cup_{k\geq 0} {\rm Crit}^\pm_k({\mc F})$.

\begin{prop}
\label{prop:singsolprop}
Let $p_\pm =( x_\pm, \eta_\pm) \in {\rm Crit}({\mc F})$ with $p_-\in {\rm Crit}_k({\mc F})$, $p_+ \in {\rm Crit}_{k-1}({\mc F})$. A fast-slow solution ${\ms X}= \left( p_0, \sigma_1, p_1, \sigma_2, p_2, \ldots, p_{n-1}, \sigma_n, p_n \right)$ from $p_-$ to $p_+$ has the following properties.
\begin{enumerate}
\item If $\sigma_i$ is a fast solution from $p_{i-1}=(x_{p_{i-1}},\eta_{p_{i-1}})$ to $p_{i}=(x_{p_{i}},\eta_{p_{i}})$ (of course $\eta_{p_i}=\eta_{p_{i-1}}$), then ${\rm index}\, (x_{p_i}, f_{\eta_{p_i}}) = {\rm index}\, (x_{p_{i-1}}, f_{\eta_{p_{i-1}}})$, unless:
\begin{enumerate}
\item $p_i\in {\mc C}_{\mc F}^{sing}$, in which case ${\rm index}\, (x_{p_i}, f_{\eta_{p_i}}) = {\rm index}\, (x_{p_{i-1}}, f_{\eta_{p_{i-1}}})-1$.
\item $i=1$, $p_0 = p_- \in {\rm Crit}_k^+({\mc F})$, in which case $\sigma_1$ is fast (so $\eta_{p_1}=\eta_p$) and ${\rm index}\, (x_{p_1}, f_{\eta_{p_1}}) = {\rm index}\, (x_-, f_{\eta_-})-1$.
\item $i=n$, $p_n = p_+ \in {\rm Crit}_{k-1}^-({\mc F})$, in which case $\sigma_n$  is fast  (so $\eta_q= \eta_{p_{n-1}}$) and ${\rm index}\, (x_+, f_{\eta_+}) = {\rm index}\, (x_{p_{n-1}}, f_{\eta_{p_{n-1}}})-1$.
\end{enumerate}
\item All slow solutions appearing in ${\ms X}$ are regular.  In particular, they are nontrivial.
\item For $i=1,\ldots,n-1$, $p_i \notin {\rm Crit}({\mc F})$.
\item The fast solutions in ${\ms X}$ are handle-slides or cusp solutions, except for $\sigma_1$ when  $p_- \in {\rm Crit}_k^+({\mc F})$ and $\sigma_n$ when $p_+ \in {\rm Crit}_{k-1}^-({\mc F})$.
\end{enumerate}
In addition:
\begin{enumerate}
\item[(I)] Suppose $p_- \in {\rm Crit}_k^+({\mc F})$ and  $p_+ \in {\rm Crit}_{k-1}^+({\mc F})$.  Then the odd $\sigma_i$ are fast, and $n$ is even.
\item[(II)] Suppose $p_- \in {\rm Crit}_k^+({\mc F})$ and  $p_+ \in {\rm Crit}_{k-1}^-({\mc F})$.  Then the odd $\sigma_i$ are fast, and $n$ is odd.
\item[(III)] Suppose $p_- \in {\rm Crit}_k^-({\mc F})$ and  $p_+ \in {\rm Crit}_{k-1}^+({\mc F})$.
Then the odd $\sigma_i$ are slow, and  $n$ is odd.
\item[(IV)] Suppose $p_- \in {\rm Crit}_k^-({\mc F})$ and  $p_+ \in {\rm Crit}_{k-1}^-({\mc F})$.
Then the odd $\sigma_i$ are slow, and  $n$ is even.
\end{enumerate}
\end{prop}

\begin{proof}
We prove (1)--(4) separately for the cases (I) and (II). The proofs for cases (III) and (IV) are similar so we omit them. We remark that in one of the cases, case (III), we can have $n=1$, which means ${\ms X}$ may contain a single slow solution and no fast solutions.

Proof in case (I).  In this case, ${\rm index}\, (x_-, f_{\eta_-})=k$ and ${\rm index}\, (x_+, f_{\eta_+})=k-1$.  Since $p_-$ is an attractor for the slow equation, the first solution $\sigma_1$ must be fast, so the odd $\sigma_i$ are fast.  By Proposition \ref{prop:indexchange} (2), $p_1 \notin {\rm Crit}({\mc F}) \cup  {\mc C}_{\mc F}^{sing}$ and ${\rm index}\, (x_{p_1},f_{\eta_{p_1}})\le k-1$.  Consider the portion of the fast-slow solution from $p_1$ to $q$.  Then Proposition \ref{prop:indexequal} implies that  ${\rm index}\, (x_{p_1},f_{\eta_{p_1}})= k-1$, conclusion (2) of  Proposition \ref{prop:singsolprop}  holds, and $p_i \notin {\rm Crit}^\pm({\mc F})$ for $i=2,\ldots,n-1$ so (3) holds.  Proposition \ref{prop:indexequal} also implies that all fast solutions except the first are handle-slides or cusp solutions, so (4) holds.

Since $p_+$ is an attractor for the slow equation, {\it a priori} the last solution could be slow or fast.  However, if the last solution were fast, from Proposition \ref{prop:indexchange} (2) we would have ${\rm index}\, (x_+, f_{\eta_+})<{\rm index}\, (x_{p_{n-1}},f_{\eta_{p_{n-1}}})=k-1$, contradiction.  Therefore the last solution is slow, so $n$ is even. Then (1) and (4) follow from Proposition \ref{prop:indexequal}.

Proof in case (II).  As with case (I), ${\rm index}\, (x_-, f_{\eta_-})=k$; $\sigma_1$ must be fast, so the odd $\sigma_i$ are fast; $p_1 \notin {\rm Crit}({\mc F}) \cup  {\mc C}_{\mc F}^{sing}$; and ${\rm index}\, (x_{p_1},f_{\eta_{p_1}})\le k-1$. Since $p_+$ is a repeller for the slow equation, ${\rm index}\, (x_+, f_{\eta_+})=k-2$, and the last solution must be fast. Then $p_{n-1} \notin  {\rm Crit}({\mc F}) \cup  {\mc C}_{\mc F}^{sing}$, and by Proposition \ref{prop:indexchange} (2), ${\rm index}\, (x_{p_{n-1}},f_{\eta_{p_{n-1}}})\ge k-1$.  As in the previous argument, for the portion of the fast-slow solution from $p_1$ to $p_{n-1}$, apply Proposition \ref{prop:indexequal}; we obtain (1)--(4). Since the first and last solutions are fast,  $n$ is odd.
\end{proof}

\subsection{The chain complex}

Now let $p, q \in {\rm Crit} ({\mc F})$. We denote by $\wt{\mc N}^0 (p, q)$ the space of all fast-slow solutions from $p$ to $q$ and ${\mc N}^0(p, q)$ the space of all fast-slow orbits from $p$ to $q$.

\begin{prop}
If $p, q\in {\rm Crit} ({\mc F})$ and ${\rm index}\, (q, {\mc F})= {\rm index}\, (p, {\mc F}) - 1$, then $\#{\mc N}^0(p, q) < +\infty$.
\end{prop}

\begin{proof}
By Proposition \ref{prop:singsolprop}, we only need to show that there are finitely many handle-slides or cusp orbits contained in ${\mc F}^{-1} \left(  [ {\mc F}(q), {\mc F}(p)] \right)$. But this follows from our description of ${\mc C}_{\mc F}$ at the beginning of Section \ref{section4} and the transversality assumption (A10), (A11), (A13).
\end{proof}

Now we define a chain complex $({\mc C}^0, \partial^0)$ of ${\mb Z}_2$-modules. It has the same generators and gradings as ${\mc C}^\lambda$. Its boundary operator $\partial^0: {\mc C}^0_k  \to  {\mc C}^0_{k-1}$ is defined by 
\begin{align}
\langle \partial^0 p, q \rangle = \# {\mc N}^0(p, q) \ {\rm mod}\ 2.
\end{align}

The main theorem of this section and of the paper is:

\begin{thm}\label{thm24} $({\mc C}^0, \partial^0)$ is a chain complex, i.e., $\partial^0 \circ \partial^0 = 0$. Moreover, the homology $H_k( {\mc C}^0, \partial^0)$ is canonically isomorphic to $H_{k-1}( \mu^{-1}(0), {\mb Z}_2)$.
\end{thm}

The proof of this theorem is carried out in the following way. Let $p, q\in {\rm Crit}({\mc F})$. First we prove that for small $\lambda>0$, any orbit in ${\mc M}^\lambda(p, q)$ is close to a fast-slow orbit in ${\mc N}^0(p, q)$ (the compactness theorem, Theorem \ref{compactnessthm}). Then we restrict to the case ${\rm index}(p, {\mc F}) - {\rm index} (q, {\mc F}) = 1$ and show that for each small $\lambda>0$, each orbit in ${\mc N}^0(p, q)$ is close to exactly one orbit in ${\mc M}^\lambda(p, q)$ (the gluing theorem, Theorem \ref{thm_glue}).  The resulting  one-to-one correspondence implies that $({\mc C}^0, \partial^0)$ is isomorphic to $({\mc C}^\lambda, \partial^\lambda)$ for small $\lambda \in \Lambda^{reg}$. Then by Corollary \ref{cor32}, we obtain the isomorphism.

\subsection{The compactness theorem}\label{subsection_compactness}

\begin{defn}
\label{defn_convergence}
Let $p, q\in {\rm Crit} ({\mc F})$. Let $\lambda_\nu \in {\mb R}^+$ be a sequence such that $\lim_{\nu\to \infty} \lambda_\nu=0$. We say that a sequence of (parametrized)  $\wt{p}_\nu = ( \wt{x}_\nu, \wt\eta_\nu) \in \wt{\mc M}^{\lambda_\nu} ( p, q)$ {\em converges} to a parametrized fast-slow solution $${\ms X}= \left(p_0, \sigma_1, p_1, \sigma_2, \ldots, p_{n-1}, \sigma_n, p_n \right) \in \wt{\mc N}^0 (p, q),$$ if the following hold.
\begin{enumerate}
\item For each fast solution $\sigma_i$ contained in ${\ms X}$, there exists $t_{i, \nu}\in {\mb R}$ such that the sequence of maps
$\wt{p}_\nu( \cdot+ t_{i, \nu})$ converges to $\sigma_i$ in the $C^\infty_{loc}$-topology.
%\item For any slow solution $\sigma_j$ contained in ${\ms X}$, there exists a sequence of intervals $I_{j,\nu} \subset {\mb R}$ such that 
%$$\lim_{\nu \to \infty} d( I_{j, \nu}, t_{i, \nu}) = \infty$$
%for any $t_{i, \nu}$ satisfying the first condition for any fast solution $\sigma_i$, and such that
%$$\lim_{\nu \to \infty} \wt{p}_\nu(I_{j, \nu}) = \sigma_j$$ in the Hausdorff topology.
\item  For each slow solution $\sigma_j$ contained in ${\ms X}$, there exists a sequence of intervals $I_{j, \nu} \subset {\mb R}$ that satisfies the following two conditions.
\begin{enumerate}
\item For every $t_{i, \nu}$ in the first condition, we have $$\lim_{\nu \to \infty} d(I_{j, \nu}, t_{i,\nu}) = \infty.$$ \item $\lim_{\nu \to \infty} \wt{p}_\nu(I_{j, \nu}) = \sigma_j$ in the Hausdorff topology.
\end{enumerate}
\end{enumerate}
\end{defn}
It is easy to see that the limit is unique in ${\mc N}^0(p, q)$ and only depends on the sequence of orbits $[ \wt{p}_\nu ]$, not the representatives. Moreover, the image of $\wt{p}_\nu$ converges to the image of ${\ms X}$ in the Hausdorff topology.

In the remainder of this section we prove the following theorem.

\begin{thm}\label{compactnessthm}
Suppose $p, q \in {\rm Crit} ({\mc F})$ and $\lambda_\nu\to 0^+$ be a sequence of real numbers. Suppose $\wt{p}_\nu= (\wt{x}_\nu, \wt\eta_\nu) \in \wt{\mc M}^{\lambda_\nu} ( p , q)$. Then there exists a subsequence (still indexed by $\nu$), and a fast-slow solution ${\ms X}= \left( p_0, \sigma_1, p_1, \sigma_2, \cdots, p_{n-1}, \sigma_n, p_n \right) \in \wt{\mc N}^0(p, q)$, such that 
$\wt{p}_\nu$ converges to ${\ms X}$ in the sense of Definition \ref{defn_convergence}.
\end{thm}

A key point in proving this theorem is to use the energy control to bound the number of pieces of fast orbits appearing in the limit. This is similar to the proof of the Gromov compactness theorem for $J$-holomorphic curves (see, for example, \cite{McDuff_Salamon_2004}) where there is a lower bound of energy for any nontrivial bubble. After finding all the fast orbits in the limit, each adjacent pair of them is connected by a slow orbit. Before the first fast orbit and after the last fast orbit, one may or may not need to add nontrivial slow ones, depending on the types of $p$ and $q$ (i.e., repeller or attracter of the slow flow).

From now on the sequence $\wt{p}_\nu$ is given and we are free to take subsequences as many times as we want.

\subsubsection{The limit set}

For any subsequence $\nu'$ of $\nu$ and sequence of intervals $I_{\nu'}\subset {\mb R}$, set
\begin{align}
L_\infty(\nu', I_{\nu'})= \bigcap_{k\geq 1} \ov{\bigcup_{\nu'\geq k} \wt{p}_{\nu'} (I_{\nu'})}\subset M\times {\mb R}.
\end{align}

\begin{lem}\label{lemma36}
The following are true.
\begin{enumerate}
\item[(I)] $(x, \eta) \in L_\infty( \nu', I_{\nu'})$ if and only if there exists a subsequence $\nu''$ of $\nu'$ and a sequence of numbers $t_{\nu''}\in I_{\nu''}$, such that $\lim_{\nu''\to \infty} \wt{p}_{\nu''}(t_{\nu''}) = (x, \eta)$.

\item[(II)] $L_\infty(\nu', I_{\nu'})$ is compact.

\item[(III)] If there exist $s_{\nu'}\in I_{\nu'}$ such that $\lim_{\nu' \to \infty} \wt{p}_{\nu'}(s_{\nu'})$ exists, then $L_\infty(\nu', I_{\nu'})$ is connected.
\end{enumerate}
\end{lem}

\begin{proof} The first statement and the closedness are by definition, and compactness follows from Lemma \ref{lemma23}. It remains to show the connectedness. If $L_\infty(\nu', I_{\nu'})$ is not connected, then there are two nonempty closed subset $L_1, L_2\subset M\times i{\mb R}$ such that $L_\infty(\nu', I_{\nu'})= L_1\cup L_2$, and $L_1, L_2$ has a nonzero distance. Suppose $\lim_{\nu'\to \infty} \wt{p}_{\nu'} (s_{\nu'})\in L_1$.  Then there exists a subsequence (still indexed by $\nu'$) and $w_{\nu'}\in I_{\nu'}$ such that $d( \wt{p}_{\nu'}(w_{\nu'}), L_2) \to 0$. Then there exists $w_{\nu'}' \in [s_{\nu'}, w_{\nu'}]$ such that $d(\wt{p}_{\nu'}(w_{\nu'}'), L_1)= d(\wt{p}_{\nu'}(w_{\nu'}'), L_2)>0$ for the same subsequence. Then there is a subsequence of $\wt{p}_{\nu'}(w_{\nu'}')$ that converges to a point not in $L_1$ and $L_2$, which contradicts $L_\infty(\nu', I_{\nu'})= L_1\cup L_2$.
\end{proof}

The following lemma is obvious.
\begin{lem}\label{lemma47}
Suppose we have two sequences of numbers, $s_\nu< t_\nu$ with \newline $\lim_{\nu \to \infty} \wt{p}_\nu(s_\nu) =a$, $\lim_{\nu \to \infty} \wt{p}_\nu( t_\nu) = b$. Then ${\mc F}( a) \geq {\mc F}(b)$, and, for any $z\in L_\infty( \nu, [s_\nu, t_\nu])$, ${\mc F}(z) \in [ {\mc F}(a), {\mc F}(b) ]$.
\end{lem}

\subsubsection{Identify all fast orbits}

By the compactness of $L_\infty(\nu, {\mb R})$ and the fact that $\lambda_\nu \to 0$, for any sequence $t_\nu\in {\mb R}$, there exists a subsequence such that on any finite interval, $\wt\eta_\nu(\cdot+ t_\nu)$ converges to a constant function $\eta_1$.  Then, by the usual argument of Morse theory, there is a subsequence of $\wt{p}_\nu( \cdot + t_\nu)$ converging to a fast solution $\wt{y}:=(y_1, \eta_1): {\mb R}\to M\times {\mb R}$ with ${\rm Im} \, \wt{y}_1 \subset M \times \{ \eta_1\}$, in $C^\infty_{loc}$-topology. So $L_\infty( \nu, {\mb R})$ is the union of fast orbits (which could contain constant orbits). We first prove that for a suitable subsequence it only contains finitely many nonconstant ones.

\begin{prop} There exists a subsequence (still indexed by $\nu$ without loss of generality), and $t_{1, \nu}, \ldots, t_{n, \nu}\in {\mb R}$ ($n$ could be zero), satisfying
\begin{enumerate}
\item[(I)] $t_{1, \nu}< t_{2, \nu} <\cdots< t_{n, \nu}$, $\lim_{\nu\to \infty} | t_{i, \nu}- t_{j, \nu}| =\infty$ for all $i\neq j$;

\item[(II)] $\wt{p}_\nu( \cdot+ t_{i, \nu})$ converges to a (nonconstant) fast solution $\wt{y}_i:= (y_i, \eta_i): {\mb R}\to M \times {\mb R}$ in the $C^\infty_{loc}$-topology;

\item[(III)] for any sequence $s_\nu\in {\mb R}$ such that $\lim_{\nu\to \infty} |s_\nu- t_{i, \nu}| = \infty$ for all $i$, any convergent subsequence of $\wt{p}_\nu(s_\nu)$ has limit in ${\mc C}_{\mc F}$.
\end{enumerate}
\end{prop}

\begin{proof} If $L_\infty( \nu, {\mb R}) \subset {\mc C}_{\mc F}$, then nothing has to be proved. Suppose it is not the case. Then we take $\delta, \epsilon>0$ as in Lemma \ref{lemma_smallenergy} and small enough, and $r>0$ small enough, such that the following are true.
\begin{itemize}
\item All $\epsilon$-short orbits are contained in $\cup_{ w \in {\mc C}_{\mc F}^{sing}} B_r(w)$.

\item All nontrivial fast orbits which are not $\epsilon$-short have energy at least $\delta$.

\item For each $w\in {\mc C}_{\mc F}^{sing}$, ${\mc C}_{\mc F}\cap B_{2r}(w)$ can be parametrized using the canonical parametrization as in (\ref{eqn510}), so that ${\mc F}$ restricted to both components of ${\mc C}_{\mc F}\cap B_{2r}(w) \setminus \{w\}$ is monotonic.
\end{itemize}

Let $N= \lfloor {E\over \delta} \rfloor$, where $E= {\mc F}(p_-)- {\mc F}(p_+)$ is the total energy. It is easy to prove by induction that there exists $k\leq N$ (possibly zero), a subsequence $\nu'$ of the original sequence, a sequence of numbers $ \left\{ t_{i,\nu'}\right\}_{i=1, \ldots, k}$, and fast solutions $\wt{y}_i:=( y_i, \eta_i)$ satisfying
\begin{enumerate}
\item $\lim_{\nu' \to \infty} t_{i+1, \nu'}- t_{i, \nu'}=+\infty$, for $i=1, \ldots, k-1$;

\item $\lim_{\nu' \to\infty} \wt{p}_{\nu'}(t_{i, \nu'} + \cdot) = (y_i, \eta_i)$ in $C^\infty_{loc}$-topology;

\item each $\wt{y}_i$ has energy $\geq \delta$;

\item any other nonconstant fast orbit contained in $L_\infty(\nu', {\mb R})$ is $\epsilon$-short.
\end{enumerate}

We replace $\nu'$ by $\nu$ for simplicity. Let $N'=\# {\mc C}_{\mc F}^{sing}$. We can continue the induction to find all $\epsilon$-short orbits, and we claim that the induction stops at finite time and we can find at most $N'(N+1)$ $\epsilon$-short orbits. Suppose not, then there exists a $w \in {\mc C}_{\mc F}^{sing}$ and a subsequence, still indexed by $\nu$, and sequence of numbers
\begin{align}
s_{1,\nu}< s_{2, \nu}<\cdots< s_{N+2, \nu} 
\end{align}
such that for each $j$, $\wt{p}_\nu(s_{j, \nu}+ \cdot)$ converges in $C^\infty_{loc}$ to an $\epsilon$-short solution $\wt{z}_j := ( z_j, \xi_j)$ whose orbit is contained in $B_r( w )$. Moreover, there exists $j\in \{1, \ldots, N+1\}$ such that $L_\infty( \nu, [s_{j, \nu}, s_{j+1, \nu}])$ contains no orbits that are not $\epsilon$-short. Then we take $[t_{j, \nu}, t_{j+1, \nu} ] \subset [s_{j, \nu}, s_{j+1 ,\nu}]$ such that
\begin{multline*}
\lim_{\nu\to \infty} \wt{p}_\nu ( t_{j, \nu}) = \lim_{t\to + \infty} \wt{z}_j(t) := a_+ \in {\mc C}_{\mc F},\\ \lim_{\nu\to \infty} \wt{p}_\nu ( t_{j+1, \nu}) = \lim_{t\to -\infty} \wt{z}_{j+1}(t) := b_- \in  {\mc C}_{\mc F}
\end{multline*}
and denote $\lim_{t\to -\infty} \wt{z}_j(t) = a_-$, $\lim_{t\to +\infty} \wt{z}_{j+1}(t)= b_+$. Suppose \newline $L_\infty( \nu, [t_{j, \nu}, t_{j+1, \nu}])$ contains some other $\epsilon$-short orbit contained in $B_r(w)$.  It can contain only finitely many, because those orbits must start from the arc between $a_-$ and $b_-$ and ends at the arc between $a_+$ and $b_+$; their energy therefore has a nonzero lower bound. Hence, by taking a subsequence if necessary, and chooing subintervals of $[t_{j, \nu}, t_{j+1, \nu}]$, we may assume that $T= L_\infty( \nu, [t_{j, \nu}, t_{j+1, \nu}])$ contains no other $\epsilon$-short orbit that is also contained in $B_r( w )$. This implies that
$$T \subset {\mc C}_{\mc F} \cup \bigcup_{v \in {\mc C}_{\mc F}^{sing},\ v\neq w} B_r(v).$$
Indeed, $T$ is contained in a connected component of the latter set.

Now by the local behavior of ${\mc F}|_{{\mc C}_{\mc F}}$ and the monotonicity requirement in Lemma \ref{lemma47}, we see that $w \notin T$. And for the same reason, one component of ${\mc C}_{\mc F}\cap (B_{2r}(w)\setminus B_r(w))$ is disjoint from $T$. This contradicts  the connectedness of $T$.
\end{proof}

\subsubsection{Identify all slow orbits}

One can not go from $p$ to $q$ only through fast orbits because $\eta_p \neq \eta_q$ according to our assumption (A12). Hence in the limit slow orbits appear. The identification of them is easy intuitively, because ${\mc C}_{\mc F}$ is 1-dimensional and our limit object should be connected.

\begin{prop} Assume the sequence $\wt{p}_\nu$ and $\{t_{i, \nu} \}_{i=1, \ldots, n}$ satisfy the conditions of the last proposition. Let $t_{0, \nu}= -\infty $ and $t_{n+1, \nu}=+\infty$. Then there exists a subsequence (still indexed by $\nu$), and slow orbits $\gamma_i \subset {\mc C}_{\mc F}$,
$i=0, \ldots, n$, satisfying the following.
\begin{enumerate}
\item[(I)] For each $i=1, \ldots, n$, $\gamma_i$ is from $\wt{y}_i(+\infty)$ to $\wt{y}_{i+1}(-\infty)$.

\item[(II)] If $p\in {\rm Crit}^+({\mc F})$, then $\gamma_0$ is the single point $p$; if $q\in {\rm Crit}^-({\mc F})$, then $\gamma_n$ is the single point $q$.

\item[(III)] For each $i=1, \ldots, n-1$, there exist a sequence of intervals $(a_{i, \nu}, b_{i, \nu}) \subset (t_{i-1, \nu}, t_{i, \nu})$ such that $\lim_{\nu\to \infty} |a_{i, \nu}- t_{i-1, \nu}| = \lim_{\nu \to \infty} | b_{i, \nu}- a_{i, \nu}| = \lim_{\nu \to \infty} | t_{i, \nu} - b_{i, \nu}| = \infty$ and $\wt{p}_\nu( [a_{i, \nu}, b_{i, ,\nu}])$ converges to $\gamma_i$ in Hausdorff topology.
\end{enumerate}
\end{prop}

\begin{proof}
By the previous proposition, we can find sequences of intervals $I_{i, \nu}= [ a_{i, \nu}, b_{i, \nu}] \subset {\mb R}$, $i=1, \ldots, n-1$ such that for each $\nu$, $I_{i, \nu}\cap I_{j, \nu}=\emptyset$, for $i\neq j$, $b_{i, \nu}- a_{i,\nu}\to \infty$, and 
$$\lim_{\nu\to \infty} \wt{p}_\nu( a_{i, \nu}) = \wt{y}_i(+\infty),\ \lim_{\nu\to \infty} \wt{p}_\nu(b_{i, \nu}) = \wt{y}_{i+1}(-\infty).$$
Take $I_{0, \nu}= (-\infty, a_{0, \nu}]$, $I_{n, \nu}= [ b_{n, \nu}, +\infty)$ such that 
$$\lim_{\nu \to \infty} \wt{p}_\nu( a_{0, \nu}) = \wt{y}_1(-\infty),\ \lim_{\nu \to \infty} \wt{p}_\nu( b_{n, \nu}) = \wt{y}_n(+\infty).$$

\begin{enumerate}

\item For each $i\in \{1, \ldots, n\}$, set $y_{i,\pm}= \wt{y}_i(\pm\infty)$, and set $y_{0, +}= p$, $y_{n+1, -} = q$. For $i=0, \ldots, n$, let $\Gamma_i \subset {\mc C}_{\mc F}$ be the union of all slow orbits from $y_{i, +}$ to $y_{i+1, -}$. Since ${\mc C}_{\mc F}$ is a 1-dimensional manifold, $\Gamma_i$ is either a single point, or is one nontrivial slow orbit, or is the union of two nontrivial slow orbits. We would like to prove that $L_\infty(\nu, I_{i, \nu}) \subset \Gamma_i$. 

 If $\Gamma_i$ is the union of two slow orbits, then $\Gamma_i$ is a connected component of ${\mc C}_{\mc F}$. By the connectedness of $L_\infty(\nu, I_{i, \nu})$ the claim is true.

Suppose $\Gamma_i$ is simply-connected and $L_\infty(\nu, I_{i, \nu})$ is also simply-connected. Suppose in this case the claim is not true, then there exists a subsequence (still indexed by $\nu$) and $t_\nu\in I_{i, \nu}$ such that $\lim_{\nu \to \infty} \wt{p}_\nu( t_\nu) = z \in L_\infty( \nu, I_{i, \nu}) \setminus \Gamma_i$ and
\begin{align}
 {\mc F}( y_{i, +}) > {\mc F}(z) > {\mc F}( y_{i+1, -}).
\end{align}
But since $L_\infty(\nu, [t_\nu, b_{i, \nu}])$ and $L_\infty( \nu, [a_{i, \nu}, t_\nu])$ are connected, either the former contains $y_{i, +}$ or the latter contains $y_{i+1, -}$, either of which contradicts Lemma \ref{lemma47}.

We claim that it is impossible that $\Gamma_i$ is homeomorphic to a closed interval and $L_\infty( \nu, I_{i, \nu})$ is homeomorphic to a circle $C\subset {\mc C}_{\mc F}$. Suppose not, then Lemma \ref{lemma47} implies that $y_{i, +}$ and $y_{i+1, -}$ must be the absolute maximum and minimum of the function ${\mc F}$ restricted to this component $C$. The critical points of ${\mc F}|_C$ can be put in a cyclic order and, adjacent to $y_{i, +}$, the two local minimum are $y_{i+1, -}$ and some $z\in C \setminus \Gamma_i$. After taking a subsequence, we can find $t_\nu\in I_{i, \nu}$ such that $\lim_{\nu \to \infty} \wt{p}_\nu( t_\nu) = z$. Consider the interval $J_{i, \nu} = [ t_\nu, b_{i, \nu}]$. Then $L_\infty( \nu, J_{i, \nu})$ is homeomorphic to a closed interval and doesn't contain $y_{i, +}$. But this is impossible because then it must contain another local maximum $z'$ between $z$ and $y_{i+1, -}$ to keep the connectedness of $L_\infty( \nu, J_{i, \nu})$, while ${\mc F}(z')\notin [{\mc F}(z), {\mc F}(y_{i+1, 0}) ]$, which contradicts with Lemma \ref{lemma47}.

\item We can identify one slow orbit contained in $\Gamma_i$. More precisely, we would like to show that, there exists $\gamma_i\subset \Gamma_i$ a slow orbit from $y_{i, +}$ to $y_{i+1, -}$ and a subsequence $\nu'$ such that $L_\infty( \nu', I_{i, \nu'})= \gamma_i$. Assume it is impossible, then $\Gamma_i= \gamma_i^1\cup \gamma_i^2$, the union of two different (nonconstant) slow orbits, and we can assume that there exists a subsequence $\nu'$ such that for {\it any} further subsequence $\{\nu''\}\subset \{\nu'\}$, $\gamma_i^1 \subset L_\infty( \nu'', I_{i, \nu''})$. 

Now suppose there exists $\widetilde{z} \in L_\infty( \nu', I_{i, \nu'} ) \cap \left( \Gamma_i\setminus \gamma_i^1 \right)$. Then there exists a subsequence $\nu''$, and $t_{\nu''}^-, t_{\nu''}, t_{\nu''}^+\in I_{i, \nu''}$, $t_{\nu''}^-< t_{\nu''} < t_{\nu''}^+$, such that
$$\lim_{\nu''\to \infty} \wt{p}_{\nu''}(t_{\nu''}) = \widetilde{z}, \ \lim_{\nu''\to \infty} \wt{p}_{\nu''} (t_{\nu''}^\pm)= \widetilde{z}^{\pm} \in {\rm Int}(\gamma_i^1)$$
and ${\mc F}(y_{i, +})> {\mc F}( \widetilde{z}^{-} )> {\mc F}( \widetilde{z} ) > {\mc F}( \widetilde{z}^+ )> {\mc F} ( y_{i+1, -}) $.

Since 
$$\left( \Gamma_i \setminus \gamma_i^1 \right)\cap {\mc F}^{-1} \left(\left[ {\mc F}( \widetilde{z}^- ), {\mc F}(  \widetilde{z}^+ ) \right] \right)$$
has a nonzero distance from $\gamma_1$, we see that
$L_\infty( \nu'',  [t_{\nu''}^-, t_{\nu''}^+] )$
is disconnected, which contradicts  Lemma \ref{lemma36}. Hence we have $L_\infty( \nu', I_{i, \nu'})= \gamma_i^1$.

\item In summary, there exists a slow orbit $\gamma_i$ from $y_{i, +}$ to $y_{i+1, -}$, and a subsequence (still indexed by $\nu$), such that $L_\infty( \nu, I_{i, \nu}) = \gamma_i$. By the connectedness of $L_\infty( \nu, I_{i, \nu})$, this implies that $\wt{p}_\nu( I_{i, \nu})$ converges to $\gamma_i$ in Hausdorff topology. Conclusion (II) of the
proposition follows from the fact that ${\mc F}$ is decreasing along $\gamma_0$ and $\gamma_n$, and $p$ is attracting and $q$ is repelling in this case.
\end{enumerate}
\end{proof}

This completes the proof of the compactness theorem.

\subsection{Gluing}

In this subsection we prove the following theorem.

\begin{thm}\label{thm_glue}
Suppose $p \in {\rm Crit}_k ({\mc F})$, $q \in {\rm Crit}_{k-1} ({\mc F})$. Then, there exists $\epsilon_0>0$, such that for all $\lambda \in (0, \epsilon_0]$, there exists a bijection
\begin{align}
\Phi^\lambda: {\mc N}^0(p, q) \to {\mc M}^\lambda(p, q) 
\end{align}
such that for any ${\ms X} \in \wt{\mc N}^0(p, q)$, there exist representatives $\wt{p}^\lambda\in\wt{\mc M}^\lambda(p, q)$ of $\Phi^\lambda([{\ms X}])$ such that as $\lambda \to 0$, $\{ \wt{p}^\lambda\}$ converges to ${\ms X}$ in the sense of Definition \ref{defn_convergence}.
\end{thm}

\begin{proof}  Let $p \in {\rm Crit}_k ({\mc F})$, $q \in {\rm Crit}_{k-1} ({\mc F})$.  Theorem \ref{compactnessthm} implies that for $\lambda>0$ small, any orbit in ${\mc M}^\lambda(p, q)$ is close to a fast-slow orbit in ${\mc N}^0(p, q)$.  We must show that for any fast-slow orbit from $p$ to $q$, if $\lambda>0$ is small, then there exists a unique orbit from $p$ to $q$ that lies near the given fast-slow orbit.

The proof consists of two parts. In the first part, we construct the gluing map $\Phi^\lambda$, by using the exchange lemma (Lemma \ref{exlem}) and general exchange lemma (Lemma \ref{genexlem}). In the second part, we show that $\Phi^\lambda$ is a bijection.

We remark that the exchange lemma and general exchange lemma are for $C^r$ differential equations, and some regularity is lost in the derivative of the constructed solution with respect to $\epsilon$.  In our situation the differential equations are of class $C^\infty$, so of course each solution is of class $C^\infty$.

Now we start to construct the gluing map. We will frequently use Proposition \ref{prop:singsolprop} without citing it, and we will use the notation $W^u(p,\lambda)$, etc., from Subsection \ref{sec:msw}.

Let ${\ms X} = \left(p_0,\sigma_1,p_1,\ldots,p_{n-1},\sigma_n,p_n\right)$ be a fast-slow solution from $p$ to $q$.  For small $\lambda>0$, we will follow $W^u(p,\lambda)$ around the fast-slow solution until it meets $W^s(q,\lambda)$ transversally.  Any orbit in ${\mc M}^\lambda(p, q)$ that is near the given fast-slow solution must be in the portion of $W^u(p,\lambda)$ that we follow; uniqueness is a consequence of the transversality at the end of the proof (Step 8).

We have $p\in {\rm Crit}_k^+ ({\mc F}) \cup {\rm Crit}_k^-( {\mc F})$ and  $q\in {\rm Crit}_{k-1}^+ ({\mc F}) \cup {\rm Crit}_{k-1}^-( {\mc F})$.

{\em Step 1.} Suppose $p=(x_p,\eta_p)\in {\rm Crit}_k^+ ({\mc F})$.   Then ${\rm index}\, (x_p, f_{\eta_p})=k$, and $\sigma_1$ is a fast solution from $p$ to $p_1=(x_{p_1},\eta_{p_1})\in{\mc C}_{\mc F} \setminus  {\mc C}_{\mc F} ^{sing}$, with ${\rm index}\, (x_{p_1}, f_{\eta_{p_1}})=k-1$.  By (A13), $W^u(p,0)$ meets $W^s(p_1,0)$ transversally within $M\times\{\eta_p\}$ along $\sigma_1$.    Then $W^u(p,0)$ meets $W^s(\sigma_2,0)$ transversally in $M\times {\mb R}$ along $\sigma_1$.  (More precisely, $\sigma_2$ should be replaced by a compact portion of the complete orbit corresponding to $\sigma_2$; we shall use this sort of abuse of notation throughout this section.)
The dimension of the intersection is ${\rm dim }\;W^u(p,0) + {\rm dim }\;W^s(\sigma_2,0) -(n+1)=k+(n
-(k-1)+1)-(n+1)=1$.  

Thus $\sigma_1$ is isolated in the intersection.  The next orbit 
$\sigma_2$ is slow, and we have (i) $p_2 \in  {\mc C}_{\mc F} \setminus  {\mc C}_{\mc F}^{sing}$, or (ii) $p_2 \in {\mc C}_{\mc F}^{sing}$.  In both cases, ${\rm index}\, (x_{p_2}, f_{\eta_{p_2}})=k-1$.

{\em Step 2.} After Step 1, in case (i), 
the exchange lemma \cite{Jones_GSPT} implies that for $\lambda>0$ small, $W^u(p,\lambda)$, followed along the flow, becomes close to $W^u(\sigma_2,0)$ near $p_2$.  By {\em close} we mean close in the $C^s$-topology for some large $s$, which decreases in the course of the proof; see \ref{app:gspt}.  Notice that the dimension of $W^u(\sigma_2,0)$ is $(k-1)+1=k$ as it should be.  

The next orbit $\sigma_3$ is fast, and we have  (i) $p_3 \in  {\mc C}_{\mc F} \setminus  {\mc C}_{\mc F}^{sing}$, or (ii) $p_3 \in{\mc C}_{\mc F}^{sing}$.  In case (i), ${\rm index}\, (x_{p_3}, f_{\eta_{p_3}})=k-1$; in case (ii), ${\rm index}\, (x_{p_3}, f_{\eta_{p_3}})=k-2$.

{\em Step 3.} After Step 1, in case (ii), let $N_0$ denote the center manifold of (\ref{eqn15a})--(\ref{eqn15b}) for $\lambda=0$ at $p_2$, which has dimension 2.  The choice of center manifold is not unique; we may choose it to include the start of $\sigma_3$ (see \ref{sec:nonu}).   By the exchange lemma, for $\lambda>0$ small, $W^u(p,\lambda)$, followed along the flow, becomes close to $W^u(\sigma_2,0)$ where $\sigma_2$ enters $N_0$. Of course $W^u(\sigma_2,0)$ (dimension $k$) is transverse to $W^s(N_0,0)$ (dimension $(n-(k-1)-1)+2=n-k+2$); $n-(k-1)-1$ is the number of negative eigenvalues at $p_2$.  The dimension of the intersection is 1.  For small $\lambda>0$, $N_0$ perturbs to a normally hyperbolic invariant manifold $N_\lambda$, and $W^u(p,\lambda)$ is transverse to  $W^s(N_\lambda,\lambda)$.

The solution in $W^u(p,\lambda) \cap W^s(N_\lambda,\lambda)$ approaches a solution $p_2^\lambda(t)$ in $N_\lambda$ that is initially near $\sigma_2\cap N_0$.  The system restricted to $N_\lambda$, $\lambda\ge0$, has been analyzed in \cite{Fold}; see  \ref{app:gspt}.  The result is that $p_2^\lambda(t)$ leaves $N_\lambda$ close to $\sigma_3\cap N_0$.  Then the general exchange lemma implies that $W^u(p,\lambda)$ becomes close to the restriction of $W^u(N_0,0)$ to $\sigma_3\cap N_0$, i.e., $W^u(p_2,0)$, as $W^u(p,\lambda)$ exits a neighborhood of $N_\lambda$. $W^u(p_2,0)$ has dimension $(k-1)+1=k$ as it should.  The orbit $\sigma_3$ is fast,  $p_3 \in {\mc C}_{\mc F} \setminus  {\mc C}_{\mc F}^{sing}$, and ${\rm index}\, (x_{p_3}, f_{\eta_{p_3}})=k-1$.

{\em Step 4.} After Step 2, in case (i),  $W^u(p,\lambda)$ is close to $W^u(\sigma_2,0)$ near $p_2$, and by (A10),  $W^u(\sigma_2,0)$ is transverse to $W^s(\sigma_4,0)$ along $\sigma_3$.  Continuation from here is like continuation after Step 1, described in Steps 2--3.

{\em Step 5.} After Step 2, in case (ii), we note that  by (A13), $W^u(p_2,0)$ (dimension $k-1$) is transverse to $W^s(p_3,0)$ (dimension $n-(k-2)$) within $M\times\{\eta_{p_3}\}$ along $\sigma_3$.  Let $N_0$ denote the center manifold of (\ref{eqn15a})--(\ref{eqn15b}) for $\lambda=0$ at $p_3$, which has dimension 2; we choose it to include the end of $\sigma_3$ (see \ref{sec:nonu}).  Then $W^u(\sigma_2,0)$ (dimension $k$) is transverse to $W^s(N_0,0)$ (dimension $(n-(k-2)-1)+2=n-k+3$) along $\sigma_3$.  The intersection is 2-dimensional and consists of solutions that track an open set of solutions in $N_0$ around $\sigma_3$. 

{\em Step 6.} After Step 3, we note that by (A13), $W^u(p_2,0)$ (dimension $k$) meets $W^s(p_3,0)$ (dimension $n-(k-1)$) transversally within $M\times \{\eta_{p_3}\}$ along $\sigma_3$.    As in Step 1, $W^u(p_2,0)$ meets $W^s(\sigma_4,0)$ transversally in $M\times {\mb R}$ along $\sigma_3$, and $W^u(p,\lambda)$ is close to $W^u(p_2,0)$ near $\sigma_3$.  The next orbit $\sigma_4$ is a slow solution from $p_3$ to (i) $p_4 \in  {\mc C}_{\mc F} \setminus  {\mc C}_{\mc F}^{sing}$, or (ii) $p_4 \in {\mc C}_{\mc F}^{sing}$.   In both cases, ${\rm index}\, (x_{p_4}, f_{\eta_{p_4}})=k-1$.
Continuation from here is like continuation after Step 1, described in Steps 2--3.

{\em Step 7.}
After Step 5, for
 small $\lambda>0$, the tracked orbits include an open set $U_\lambda$ in $N_\lambda$, the perturbation of $N_0$, that lies above an open set $U_0$ in $N_0$ that contains  $\sigma_4\cap N_0$ in its interior. The general exchange lemma implies that $W^u(p,\lambda)$ is close to the restriction of $W^u(N_0,0)$ to $U_0$, i.e., to $W^u(\sigma_4\cap N_0)$, as it exits a neighborhood of $N_\lambda$.  Continuation from here is like continuation after Step 2, described in Steps 4--5.

{\em Step 8.}  Continuation proceeds using analogs of the steps previously described until $W^u(p)$ arrives near $p_{n-1}$.

If $q\in {\rm Crit}_k^+( {\mc F})$, then $\sigma_n$ is slow, and ${\rm index}\, (x_{p_{n-1}}, f_{\eta_{p_{n-1}}})={\rm index}\, (x_q, f_{\eta_q})=k-1$.  $W^u(p,\lambda)$ arrives  near $p_n$ along the slow orbit $\sigma_n$, close to $W^u(\sigma_n,0)$, which has dimension $k$.  $W^u(\sigma_n,0)$ is transverse to $W^s(\sigma_n,0)$, with dimension $n-(k-1)+1=n-k+2$.  Therefore $W^u(p,\lambda)$ is transverse to $W^s(q,\lambda)$, which is close to $W^s(\sigma_n,0)$. The intersection has dimension 1 and gives the solution from $p$ to $q$.

If $q\in {\rm Crit}_k^-( {\mc F})$, then $\sigma_n$ is a fast orbit from $p_{n-1} \in {\mc C}_{\mc F} \setminus  {\mc C}_{\mc F}^{sing}$ to $q \in {\mc C}_{\mc F} \setminus  {\mc C}_{\mc F}^{sing}$; ${\rm index}\, (x_{p_{n-1}}, f_{\eta_{p_{n-1}}})=k-1$ and ${\rm index}\, (x_q, f_{\eta_q})=k-2$.  By (A11), $W^u (p_{n-1},0)$ meets $W^s(q, 0)$ transversally in $M\times\{\eta_q\}$ along $\sigma_n$.  Therefore $W^u(\sigma_n,0)$ meets $W^s(q,0)$ transversally in $M\times{\mb R}$ along $\sigma_n$.  Since $W^u(p,\lambda)$ is close to $W^u(\sigma_n,0)$, $W^u(p,\lambda)$ meets $W^s(q,\lambda)$ transversally  in $M\times{\mb R}$ near $\sigma_n$.  The intersection has dimension one and gives the solution from $p$ to $q$.

{\em Step 9.} Suppose $p=(x_p,\eta_p)\in {\rm Crit}_k^-( {\mc F})$.   Then ${\rm index}\, (x_p, f_{\eta_p})=k-1$, and $\sigma_1$ is a slow solution from $p$ to $p_1=(x_{p_1},\eta_{p_1})$.  Either (i)  $p_1\in{\mc C}_{\mc F} \setminus  {\mc C}_{\mc F} ^{sing}$, or (ii) $p_1\in{\mc C}_{\mc F} ^{sing}$.  In either case, 
${\rm index}\, (x_{p_1}, f_{\eta_{p_1}})=k-1$.

In case (i),  $W^u(p,\lambda)$ is close to $W^u(\sigma_1,0)$.  Step 3 above describes how to continue from $p_1$.  Case (ii) is left to the reader.

Now we show that the gluing map just constructed is a bijection. We first note that there exist small $\delta >0$ and $\lambda_0>0$ such that for all $\lambda \in (0, \lambda_0)$, the orbit we constructed is the unique one in ${\mc M}^\lambda(p, q)$ such that its Hausdorff distance from the fast-slow orbit $[{\ms X}]$ is less than $\delta$.  This is a general fact about exchange lemma constructions; see for example \cite{Jones_Tin}, p. 1021.

Thus the gluing map is a bijection provided it is surjective. We claim that there exists $\lambda_0>0$ such that the gluing map is surjective for all $\lambda \in (0, \lambda_0)$. If not, then there exist a sequence $\lambda_i \to 0$ and a sequence of orbits $[\wt{x}_i] \in {\mc M}^{\lambda_i}(p, q)$ such that for each $i$, $[\wt{x}_i]$ is not in the image of $\Phi^{\lambda_i}$. Then, by the compactness theorem (Theorem \ref{compactnessthm}), without loss of generality, $[\wt{x}_i]$ converges to a fast-slow orbit $[{\ms Y}]$ in the sense of Definition \ref{defn_convergence}. In particular, this sequence converges in the Hausdorff topology. Then for large $i$, $\Phi^{\lambda_i}([{\ms Y}])$, which is the unique orbit in ${\mc M}^{\lambda_i}(p, q)$ in a small Hausdorff neighborhood of $[{\ms Y}]$, must be $[\wt{x}_i]$. This contradiction proves the claim. 
\end{proof}

\section{Acknowedgements}

The first author's work was partially supported by NSF under award DMS-1211707. 
The second author would like first to thank his advisor Professor Gang Tian, for his help and encouragement, guidance and corrections. He also would like to thank Urs Frauenfelder for many helpful discussions, to thank Robert Lipshitz and Dietmar Salamon for their interest in this work, to thank Hongbin Sun for answering topology questions. At the early stage of this project, the second author was funded by Professor Helmut Hofer for the Extramural Summer Research Project of Princeton Univeristy in Summer 2011, and he would like to express his gratitude for Professor Hofer's generosity.

\appendix
\section{Exchange lemmas\label{app:gspt}}

%Readers familiar with geometric singular perturbation theory may wish to skip most of this appendix; however, the paragraph before Theorem \ref{genexlem} and that theorem may be unfamiliar, and the extension of that theorem described toward the end of this section, while just an observation, is new and essential to this paper.

Let $Z$ be a manifold, let $\dot p=f(p,\epsilon)=f_\epsilon(p)$ be a differential equation on $Z$ with parameter $\epsilon$ (i.e., each $f_\epsilon$ is a section of $TZ$), let $\phi_t^\epsilon$ be the flow, and let $N_0$ be a compact submanifold of $Z$ that is invariant under the the flow of $\dot p=f_0(p)$.  $N_0$ is called {\em normally hyperbolic} if there is a splitting of $TZ|N_0$, $TZ|N_0= S \oplus U \oplus TN_0$, such that under $D\phi_t^0$, all vectors in $S$ shrink at a faster exponential rate than any vector in $TN_0$, and under $D\phi_{-t}$, all vectors in $U$ shrink at a faster exponential rate than any vector in $TN_0$.  (There are less restrictive definitions, but this one suffices for our purposes.)

Normally hyperbolic invariant manifolds have stable and unstable manifolds with flow-preserved fibrations, and the whole structure persists under perturbation. This structure is most easily described in local coordinates. 

Let us assume that ${\rm dim} \, N_0=m$, ${\rm dim} \, Z=n+m$, and fibers of $S$ (respectively $U$) have dimension $k$ (respectively $l$), with $k+l=n$. Near a point of $N_0$ one can choose coordinates $p=\Phi(x,y,z,\epsilon)$, $(x,y,z,\epsilon)\in\Omega_1\times\Omega_2\times(-\epsilon_0,\epsilon_0)$, $\Omega_1$ open subset of ${\mb R}^k \times {\mb R}^l$, $\Omega_2$ an open subset of ${\mb R}^m$, such that, for small $\epsilon$, $\dot p=f(p,\epsilon)$ becomes
\begin{align}
\dot x &= A(x,y,z,\epsilon)x, \label{nh1}  \\
\dot y &= B(x,y,z,\epsilon)y,  \label{nh2}   \\
\dot z &= h(z,\epsilon) + x^TC(x,y,z,\epsilon)y; \label{nh3}  
\end{align}
the matrices $A$, $B$, and $C$ are $k\times k$,  $l\times l$, and $k\times l$ respectively.  %One cannot assume that  $A(0,0,z,0)$ has eigenvalues with negative real part or that $B(0,0,z,0)$ has eigenvalues with positive real part, since the exponential convergence of vectors under the linearized flow need only occur as $t \to \pm\infty$.  Nevertheless, this is true in the examples below.

We list some facts and terminology.
\begin{enumerate}
\item  If $\dot p=f(p,\epsilon)$ is $C^{r+3}$, the coordinate change can be chosen so that the new system is $C^{r+1}$.
\item For each $\epsilon$, the subspaces $y=0$, $x=0$, and their intersection are invariant.  For fixed $\epsilon$, the set $\{(x,y,z) \ | \  x=0 \mbox{ and } y=0\}$ (dimension $m$) corresponds to part of a normally hyperbolic invariant manifold $N_\epsilon$; the set $y=0$ (dimension $m+k$) corresponds to part of the stable manifold of $N_\epsilon$, $W^s(N_\epsilon)$; and the set $x=0$ (dimension $m+l$) corresponds to part of the unstable manifold of $N_\epsilon$, $W^u(N_\epsilon)$.
\item If $(x(t),0,z(t))$ is a solution in $W^s(N_\epsilon)$, then $(0,0,z(t))$ is a solution in $N_\epsilon$; and if $(0,y(t),z(t))$ is a solution in $W^u(N_\epsilon)$, then $(0,0,z(t))$ is again a solution in $N_\epsilon$.  Each solution in $W^s(N_\epsilon)$ (respectively $W^u(N_\epsilon)$) approaches exponentially a solution in $N_\epsilon$ as time increases (respectively decreases). 
\item Given  a point $p=(0,0,z_0)$ in $N_\epsilon$, the {\em stable} (respectively {\em unstable}) {\em fiber} of $p$ is the set of all points $(x,0,z_0)$ (dimension $k$) (respectively $(0,y,z_0)$ (dimension $l$)).  For each $t$, the time-$t$ map of the flow takes fibers to fibers; in this sense the fibration is flow-invariant. Solutions that start in the stable (respectively) unstable fiber of $p$ approach the solution that starts at $p$ exponentially at $t$ increases (respectively decreases).
\item Given an $P_\epsilon \subset N_\epsilon$, we shall refer to the union of the stable (respectively unstable) fibers of points in $P_\epsilon$ as $W^s(N_\epsilon)$ (respectively $W^u(N_\epsilon))$ {\em restricted to} $P_\epsilon$. If $P_\epsilon$ is invariant, $W^s(N_\epsilon)$ restricted to $P_\epsilon$, for example, may be smaller than  $W^s(P_\epsilon)$, since the latter may include solutions in $N_\epsilon$ that approach $P_\epsilon$ at a slower exponential rate, together with points in their stable fibers.
\end{enumerate}

%\item Suppose $N_0$ is a compact manifold of equilibria of dimension $m$, and each equilibrium in $N_0$ has $k$ eigenvalues with negative real part and $l$ eigenvalues with positive real part.  Then $N_0$ is a compact normally hyperbolic invariant manifold.  The stable (respectively unstable) fiber of a point $p$ in $N_0$ is just its stable (respectively unstable) manifold.  In (\ref{nh3}), $h(z,0) \equiv 0$.

Suppose $N_0$ is a compact manifold with boundary of equilibria of dimension $m$, and each equilibrium in $N_0$  has $k$ eigenvalues with negative real part and $l$ eigenvalues with positive real part.  $N_0$ does not fit the definition of a normally hyperbolic invariant manifold that we have given, but a coordinate system as above still exists, except that $\Omega_2$ may have boundary and $N_\epsilon$ may be only locally invariant (i.e., solutions may leave through the boundary).  We shall abuse terminology and refer to the $N_\epsilon$ as normally hyperbolic invariant manifolds in this situation.

Suppose we wish to follow an $(l+1)$-dimensional manifold of solutions $M_\epsilon$ as it passes near a normally hyperbolic invariant manifold $N_\epsilon$ in a manifold $Z$, where $N_0$ is a normally hyperbolic invariant manifold of equilibria.  We choose coordinates so that the system is (\ref{nh1})--(\ref{nh3}), with $h(z,\epsilon)=\epsilon\tilde h(z,\epsilon)$.  We assume that near the point $(x^*,0,z^*)$, the $M_\epsilon$, $\epsilon\ge0$, fit together to form a $C^{r+1}$ manifold with boundary in $Z\times{\mb R}$; and we assume that
$M_0$ meets $W^s(N_0)$ transversally in the solution through  $(x^*,0,z^*)$, 
which approaches the equilibrium $(0,0,z^*)$ as $t\to\infty$. The system restricted to $N_\epsilon$ has the form $\dot z = \epsilon \tilde h(z,\epsilon)$; we assume $\tilde h(z^*,0)\neq 0$.  Let $\psi_t$ denote the flow of $\dot z = \tilde h(z,0)$, and let $z^\dagger=\psi_T(z^*)$ for some $T>0$.  Let $I=\{(\psi_t(z^*) : T-\delta<t<T+\delta\}$, a small interval around $z^\dagger$ in the orbit of $z^*$ for $\dot z = \tilde h(z,0)$.  Let $J=\{0\}\times\{0\}\times I$.  Let $V$ be a small open set around $(0,0,z^\dagger)$ in $W^u(N_0)$ restricted to $J$,
which has dimension $l+1$. Then we have:

\begin{thm}[Exchange Lemma]\label{exlem}
For small $\epsilon>0$, parts of $M_\epsilon$ fit together with $V$ to form a $C^r$ manifold in $Z\times{\mb R}$.
\end{thm}

\begin{figure}[htbp]
\centering
\includegraphics[scale=1.0]{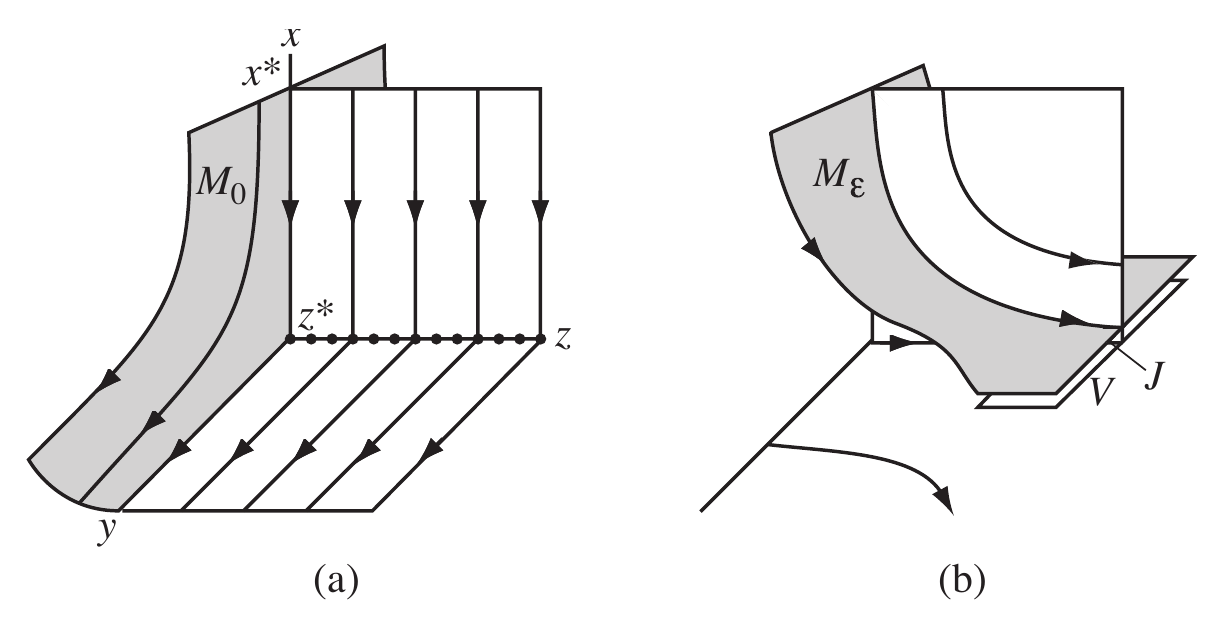}
\caption{The exchange lemma with $k=l=m=1$: (a) $\epsilon=0$, (b) $\epsilon>0$.}
\label{fig:exch}
\end{figure}

See Figure \ref{fig:exch}.  Notice that for small $\epsilon>0$, $M_\epsilon$ meets $W^s(N_\epsilon)$ transversally in a solution that tracks a solution in $N_\epsilon$ that starts at a point $(0,0,z(\epsilon))$, with $z(\epsilon)$ near $z^*$.  For $\epsilon>0$, the orbit of $\dot z = \epsilon \tilde h(z_\epsilon)$ through $z_\epsilon$ equals the orbit of $\dot z = \tilde h(z,\epsilon)$ through $z(\epsilon)$.  This orbit is close to the orbit of $\dot z =\tilde  h(z,0)$ through $z^*$.  Thus we find that portions of the orbits tracked in $N_\epsilon$ for $\epsilon>0$ limit on a curve $J$  in $N_0$ that is not an orbit of the system (\ref{nh1})--(\ref{nh3})  on $N_0$ (these orbits are just points).  Nevertheless the $M_\epsilon$ for $\epsilon>0$ become close to $W^u(N_0)$ restricted to $J$.  

Now suppose we wish to track an $(l+s+1)$-dimensional manifold of solutions $M_\epsilon$, with $0\le s\le m-1$, as it passes near a normally hyperbolic invariant manifold $N_\epsilon$ in a manifold $Z$, where $N_0$ is a normally hyperbolic invariant manifold but does not consist of equilibria. (This occurs, for example, Step 3 of the proof of Theorem \ref{thm_glue}.) We again choose coordinates so that the system is (\ref{nh1})--(\ref{nh3}).  We choose a cross-section $\tilde M_\epsilon$ to the flow within $M_\epsilon$, of dimension $l+s$ (so that  $M_\epsilon$ is the union of orbits that start in $\tilde M_\epsilon$),    and we assume that $\tilde M_0$ meets $W^s(N_0)$ transversally at  $(x^*,0,z^*)$.  For small $\epsilon\ge0$, $\tilde M_\epsilon$ meets $W^s(N_\epsilon)$ in a manifold $Q_\epsilon$ of dimension $s$; we assume that $Q_\epsilon$ projects regularly along the stable fibration to a submanifold $P_\epsilon$ of $N_\epsilon$ of dimension $s$.  We also assume that for small $\epsilon\ge0$, the vector field $(0,0,h(x,\epsilon))$ is not tangent to $P_\epsilon$.  (Of course these assumptions follow from the corresponding ones at $\epsilon=0$.)  Finally, we assume that for $\epsilon>0$, following $P_\epsilon$ along the flow for time ${\mc O}(\frac{1}{\epsilon})$ produces a submanifold $P_\epsilon^*$ of dimension $s+1$ of $N_\epsilon$, and the manifolds $P_\epsilon^*$ fit together with a submanifold $P_0^*$ of $N_0$, of dimension $s+1$, to form a $C^{r+1}$ manifold in $Z\times{\mb R}$.  This assumption typically requires that some solution of the system restricted to $N_0$ that starts in $P_0$ approaches an equilibrium.
We emphasize that, analogous to the usual exchange lemma, $P_0^*$ is not the result of following $P_0$ along the flow for $\epsilon=0$. Let $V$ be a small open neighborhood of $P_0^*$ in $W^u(N_0)$ restricted to $P_0^*$, which has dimension $l+s+1$.

\begin{thm}[General Exchange Lemma]\label{genexlem}
For small $\epsilon>0$, parts of $M_\epsilon$ fit together with $V$ to form a $C^r$ manifold in $Z\times{\mb R}$.
\end{thm}

Some technical hypotheses that are not relevant to the present paper have been omitted.  We actually need a small generalization of Theorems \ref{exlem} and \ref{genexlem}, whose proofs are essentially the same. 

We will use $C^r$-topology to measure submanifolds. If $S\subset X$ is a smooth submanifold, we say another $C^r$-submanifold $S'$(of the same dimension) is $C^r$-close to $S$, if with some smooth identification of the normal bundle of $S$ with a tubular neighborhood of $S$, $S'$ can be identified with a $C^r$-small section of the normal bundle.

In Theorem \ref{exlem}, replace the assumption that near the point $(x^*,0,z^*)$, the $M_\epsilon$, $\epsilon\ge0$, fit together to form a $C^{r+1}$ manifold with boundary, with the assumption that $M_\epsilon \to M_0$ in the $C^{r+1}$-topology.  The conclusion becomes that parts of $M_\epsilon$ converge to $V$ in the $C^r$-topology. 

In Theorem \ref{genexlem}, make the replacement in the assumptions just mentioned, and replace the assumption that the $P_\epsilon^*$, $\epsilon\ge0$, fit together to form a $C^{r+1}$ manifold with the assumption that $P_\epsilon^* \to P_0^*$ in the $C^{r+1}$-topology.  The conclusion again becomes that parts of $M_\epsilon$ converge to $V$ in the $C^r$-topology. 

These generalizations are required for the following reason.  Consider the fast-slow system 
\begin{align}
\dot z_1 &= -z_2+z_1^2,  \label{cmnm1}\\
\dot z_2 &= -\epsilon,  \label{cmnm2}
\end{align}
which is related to (\ref{cmdex})--(\ref{cmdeeta}).  See Figure \ref{fig:cm}.  For $\epsilon=0$, any compact portion $N_0$ of the curve of equilibria $z_1=-z_2^\frac{1}{2}$ is normally hyperbolic (in fact attracting).  It perturbs to the normally hyperbolic manifold $N_\epsilon$, on which system reduces to $\dot y=-\epsilon$.  For small $\epsilon>0$, a solution in or close to  $N_\epsilon$ arrives in the region $\delta<z_1<2\delta$ ($\delta>0$) along a curve given by $z_2=\rho(z_1,\epsilon)$, $\delta<z_1<2\delta$; $\rho={\mc O}(\epsilon^\frac{2}{3})$.   As $\epsilon \to 0$, this curve, which in examples with greater dimension may be $P_\epsilon^*$, approaches $z_2=0$, which in examples may be $P_0^*$, in the $C^s$-topology for any $s$, but does not fit together with $z_2=0$ to form a manifold with a high degree of differentiability in $z_1z_2\epsilon$-space.

\section{Choosing the center manifold\label{sec:nonu}}

Let us consider for concreteness a system in the form (\ref{nh1})--(\ref{nh3}), with $N_0$ two-dimensional, and the equation on $N_\epsilon$ given by (\ref{cmnm1})--(\ref{cmnm2}):
\begin{align}
\dot x &= A(x,y,z_1,z_2,\epsilon)x, \label{appb1}  \\
\dot y &= B(x,y,z_1,z_2,\epsilon)y,  \label{appb2}   \\
\dot z_1 &= -z_2+z_1^2+x^Tc_1(x,y,z_1,z_2,\epsilon)y,  \label{appb3}\\
\dot z_2 &= -\epsilon+x^Tc_2(x,y,z_1,z_2,\epsilon)y,  \label{appb4}
\end{align}
$(x,y) \in {\mb R}^k\times {\mb R}^l$. The system arises by center manifold reduction at the origin, so $A(0,0,z_1,z_2,0)$ has eigenvalues with negative real part ,and $B(0,0,z_1,z_2,0)$ has eigenvalues with positive real part.  We wish to follow a manifold $M_\epsilon$ of dimension $l+1$ as it passes $N_\epsilon$ (i.e., $z_1z_2$-space).  $M_0$ meets $W^s(N_0)$ transversally at a point $(x^*, 0.z_1^*,0)$ with $z_1^*<0$; hence the intersection includes  the semiorbit $\gamma$ that starts at  $(x^*,0, z_1^*,0)$, which approaches the origin as $t\to\infty$.  We wish to replace the center manifold $N_0$ for $\epsilon=0$ by one that contains $\gamma$.  To do this, replace all semiorbits in Figure \ref{fig:cm} that start at points $(0,0, z_1^*,z_2)$, $|z_2| <(z_1^*)^2$, with the semiorbits that start at $(x^*,0, z_1^*,z_2)$.  (The semiorbits that  start at  $(0,0, z_1^*,z_2)$ and at $(x^*, 0,z_1^*,z_2)$ both approach $(0, 0,-z_2^\frac{1}{2},z_2)$ as $t\to\infty$, and the second arrives tangent to $z_1z_2$-space.)
The result will be a new center manifold, in a smaller neighborhood of the origin, that contains $\gamma$.  The differentiability class of this manifold will decrease as we move away from the origin.  Now $N_0$ perturbs to new center manifold $N_\epsilon$ for $\epsilon>0$.

\begin{figure}[htbp]
\centering
\includegraphics[scale=0.8]{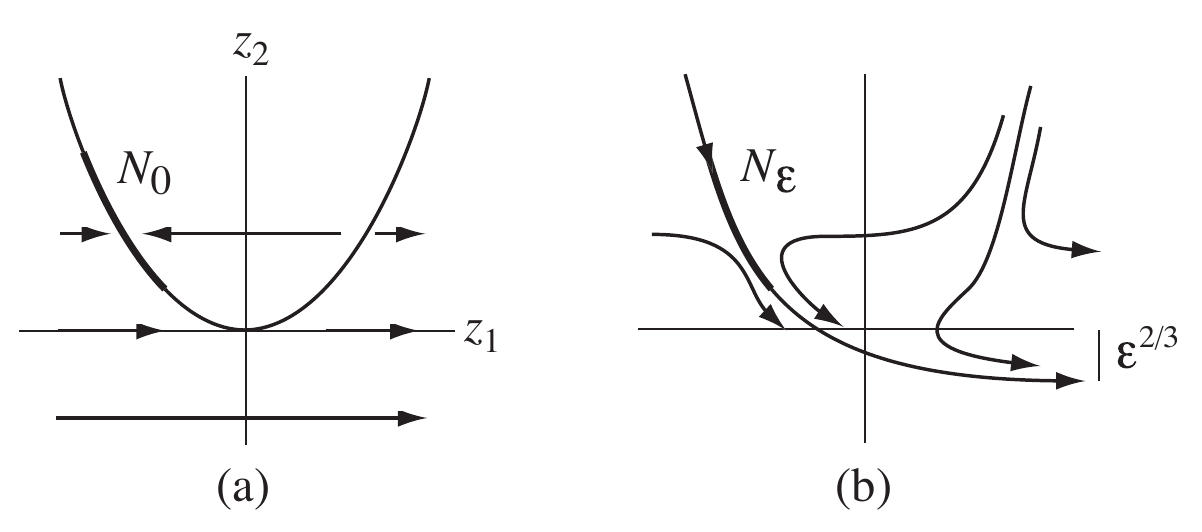}
\caption{Flow of (\ref{cmnm1})--(\ref{cmnm2}): (a) $\epsilon=0$, (b) $\epsilon>0$.}
\label{fig:cm}
\end{figure}

\bibliography{symplectic_ref}

\bibliographystyle{elsarticle-num}

\end{document}